\documentclass[10pt]{article}
\usepackage[top=1in, bottom=1in, left=1in, right=1in]{geometry}
\usepackage{amssymb,amsfonts,amsmath,amsthm,amscd,dsfont,mathrsfs,mathtools,microtype,nicefrac,pifont,soul}
\usepackage{todonotes}
\usepackage{lipsum}
\usepackage{amsfonts}
\usepackage{graphicx}
\usepackage{epstopdf}
\usepackage{algorithmic}
\usepackage{amsopn}
\usepackage{algorithmic}
\usepackage{subfigure}
\usepackage[utf8]{inputenc}
\usepackage[T1]{fontenc}   
\usepackage[linktocpage,colorlinks,linkcolor=blue,anchorcolor=blue,citecolor=blue,urlcolor=blue,pagebackref]{hyperref}
\usepackage{url}            
\usepackage{booktabs}       
\usepackage{xcolor}         
\usepackage[ruled,vlined,linesnumbered]{algorithm2e}
\usepackage{bbm}
\usepackage{bm}
\usepackage{color, colortbl}
\definecolor{LightCyan}{rgb}{0.88,1,1}
\usepackage{dirtytalk}
\usepackage{dsfont}
\usepackage{enumerate}
\usepackage{listings}
\usepackage{enumitem}
\usepackage{xspace}
\usepackage{makecell}
\usepackage{multirow, array}
\usepackage{verbatim}
\usepackage[capitalize,noabbrev]{cleveref}
\usepackage{tikz}
\usepackage{tcolorbox}
\usepackage{tablefootnote}
\usepackage[symbol]{footmisc}
\usepackage{caption}
\usepackage{wrapfig}

\newtheorem{theorem}{Theorem}

\newtheorem{lemma}{Lemma}[section]

\newtheorem{assumption}{Assumption}

\newcommand{\N}{\mathbb{N}}

\newcommand{\R}{\mathbb{R}}

\newcommand{\E}{\mathbb{E}}
\newcommand{\T}{^\mathsf{T}}
\newcommand{\G}{\mathcal{G}}
\newcommand{\F}{\mathcal{F}}

\newcommand{\norm}[1]{\left\|#1\right\|}

\newtcolorbox{thmbox}{colback=cyan!5,colframe=white}
\newtcolorbox{questionbox}{colback=red!5!white,colframe=white}
\newtcolorbox{updatebox}{colback=white,colframe=black}

\DeclareMathOperator*{\argmin}{arg\,min}
\definecolor{LightCyan}{rgb}{0.88,1,1}

\def\<#1,#2>{\langle #1,#2\rangle}
\allowdisplaybreaks

\renewcommand*{\backref}[1]{\ifx#1\relax \else Page #1 \fi}
\renewcommand*{\backrefalt}[4]{%
  \ifcase #1 \footnotesize{(Not cited.)}%
  \or        \footnotesize{(Cited on page~#2.)}%
  \else      \footnotesize{(Cited on pages~#2.)}%
  \fi
}
\title{\bf Stochastic Nested Compositional Bi-level \\ Optimization for Robust Feature Learning}

\author{Xuxing Chen\thanks{Department of Mathematics, University of California, Davis. 
  (\texttt{xuxchen@ucdavis.edu})}
\and Krishnakumar Balasubramanian\thanks{Department of Statistics, University of California, Davis. Supported by NSF Grant DMS-2053918.
  (\texttt{kbala@ucdavis.edu})}
\and Saeed Ghadimi\thanks{Department of Management Sciences, University of Waterloo.
(\texttt{sghadimi@uwaterloo.ca})}}

\begin{document}

\maketitle

\begin{abstract}
We develop and analyze stochastic approximation algorithms for solving nested compositional bi-level optimization problems. These problems involve a nested composition of $T$ potentially non-convex smooth functions in the upper-level, and a smooth and strongly convex function in the lower-level. Our proposed algorithm does not rely on matrix inversions or mini-batches and can achieve an $\epsilon$-stationary solution with an oracle complexity of approximately $\tilde{O}_T(1/\epsilon^{2})$, assuming the availability of stochastic first-order oracles for the individual functions in the composition and the lower-level, which are unbiased and have bounded moments. Here, $\tilde{O}_T$ hides polylog factors and constants that depend on $T$. The key challenge we address in establishing this result relates to handling three distinct sources of bias in the stochastic gradients. The first source arises from the compositional nature of the upper-level, the second stems from the bi-level structure, and the third emerges due to the utilization of Neumann series approximations to avoid matrix inversion. To demonstrate the effectiveness of our approach, we apply it to the problem of robust feature learning for deep neural networks under covariate shift, showcasing the benefits and advantages of our methodology in that context.
\end{abstract}



\section{Introduction}\label{sec:intro}
We study a new class of optimization problems, namely the nested compositional bi-level problems, that are given by
\begin{align}
\min_{x\in X}\;\; \Phi(x):= \Psi(x, y^*(x)), \quad\;\;\mbox{s.t.} \quad y^*(x)= \argmin_{y\in\R^q} g(x,y),\label{eq:multilevel}
\end{align}
where $X$ is a closed and convex set in $\R^p$, and $\Psi$ is a compositional function defined as $\Psi(x,y) = f_1\circ \dots \circ f_T(x,y)$. The functions $f_i(z)\coloneqq\E\left[F_i(z;\xi_i)\right] :\R^{d_i} \to \R^{d_{i-1}}$ for $i=1,2,...,T$, in the upper-level of \eqref{eq:multilevel} are assumed to be smooth. The function $g(x,y):= \E\left[G(x,y;\zeta)\right]:\R^{d_T} \to \R$ in the lower-level is also assumed to be smooth and strongly convex, with $d_T = p + q$. Our goal in this work is to develop and analyze fully-online, batch-free stochastic approximation algorithms to solve~\eqref{eq:multilevel}, given access to stochastic gradients (and function values) of the individual function. 


\textbf{Motivation.} The motivation behind our algorithm design for solving~\eqref{eq:multilevel} lies in our primary objective of developing robust feature learning methods within the realm of predictive deep learning. It is widely acknowledged that in various applications, the distribution of the testing data significantly differs from that of the training data, which is commonly referred to as covariate shift~\cite{pan2010survey,sugiyama2012machine}. Ensuring the robustness of training procedures against this shift is crucial for the effective implementation of predictive deep learning techniques in real-world scenarios. 

As a concrete example, consider solving least-squares non-parametric regression with deep neural networks (e.g., \cite{schmidt2020nonparametric}) with $Z\in\mathbb{R}^d$ being the input data and $Y\in \mathbb{R}$ being the real-valued response. We denote by $\Phi: \mathbb{R}^d \to \mathbb{R}^p$ as the feature learned by a depth $L$ neural network and $\beta \in \mathbb{R}^p$ to be the last layer of the neural network. With this notation, we could explicitly decouple the feature learning part (which is captured by $\Phi$) and the regression part (which is captured by $\beta$). The problem of learning robust features in this context could be formulated as the following distributionally robust (DR) bi-level optimization problem.

\begin{align}\label{eq:dro_minmax_form}
\underset{\Phi}{\min}~~ \underset{Q \in B(P)}{\max} ~~  \E_Q [(Y  - \langle \beta, \Phi(X)\rangle)^2] \quad\textrm{s.t.}\quad \beta = \underset{\tilde \beta \in \mathbb{R}^p}{\argmin}~~\E_P\big[ (Y  - \langle \tilde \beta, \Phi(X)\rangle)^2 \big].
\end{align}
When a coherent risk measure is chosen, the distributionally robust optimization problem at the top-level could be reformulated as a risk minimization problem~\cite{shapiro2021lectures}. In particular, when a mean semi-deviation risk measure is used, the min-max problem above can be reformulated as the following compositional minimization problem
\begin{align}
\begin{aligned}
\underset{\Phi}{\min} \big\{ \E[(Y  - \langle \beta, \Phi(X)\rangle)^2&] + \lambda (\E[\max(0,(Y  - \langle \beta, \Phi(X)\rangle)^2 - \E[(Y  - \langle \beta, \Phi(X)\rangle)^2] )^2])^{1/2} \big\}\\
&\textrm{s.t.}\quad\beta = \underset{\tilde \beta \in \mathbb{R}^p}{\argmin}~~\E_P\big[ (Y  - \langle \tilde \beta, \Phi(X)\rangle)^2 \big],\label{eq:droform}
\end{aligned}
\end{align}
which fits the general setup that we consider in~\eqref{eq:multilevel}. We refer, for example, to~\cite{yang2019multi, ruszczynski2020convergence,balasubramanian2022stochastic,zhu2023distributionally} for additional details on the reformulation. In Section~\ref{sec:experiments}, we provide simulation experiments illustrating the robustness of the above approach to certain classes of covariate shifts.

\textbf{Bias sources.} The main challenge in designing and analyzing stochastic optimization algorithms for solving~\eqref{eq:multilevel} is dealing with the various sources of bias arising in estimating the gradient. To illustrate the point, first note that under standard smoothness assumptions, it is easy to see that the gradient of $\Phi$ in~\eqref{eq:multilevel} is given by
\begin{align}
    \nabla \Phi(x) =   \textcolor{red}{\nabla_x\Psi}(x, \textcolor{blue}{y^*(x)}) -  \nabla_{xy}^2 g(x, \textcolor{blue}{y^*(x)})\cdot \underset{\text{\textcolor{violet}{Neumann series approximation }}}{\boxed{\nabla^2_{yy} g(x, \textcolor{blue}{y^*(x)})^{-1}  \textcolor{red}{\nabla_y\Psi}(x, \textcolor{blue}{y^*(x)})}}.\label{eq: hyper_grad}
\end{align} 
In the above expression, we have highlighted the following three sources of bias. The \textcolor{red}{first one} is due to the presence of nested composition of $T$ functions at the upper-level. The \textcolor{blue}{second one} is due to the presence of the bi-level structure in which there is a lack of knowledge of the exact solution to the lower problem. Moreover, calculating the Hessian inverse in \eqref{eq: hyper_grad} is computationally prohibitive even for moderate size problems. A common practice in bi-level optimization literature to avoid this matrix inversion is to  use the Neumann series based methods to approximate the Hessian inverse or its product and the gradient directly (see e.g., \cite{ghadimi2018approximation, hong2023two, ji2021bilevel, chen2021closing}). 
The \textcolor{violet}{third source} of bias thus comes from $\norm{\bar{r} - \nabla^2_{yy} g(x, \textcolor{blue}{y^*(x)})^{-1} \textcolor{red}{\nabla_y \Psi}(x, \textcolor{blue}{y^*(x)})}$, where $\bar{r}$ denotes the vector obtained by approximating the solution of the linear system $\nabla^2_{yy} g(x, \textcolor{blue}{y^*(x)})\cdot r = \textcolor{red}{\nabla_y \Psi}(x, \textcolor{blue}{y^*(x)})$. Moreover, the aforementioned sources of bias are also nested, i.e., the bias arising due to Neumann series approximation is affected by the nested structure of $\Phi$ and lack of knowledge of $y^*(x)$. A major contribution of our work lies in carefully dealing with the above three sources of bias in the design of our algorithm and its convergence analysis. 

\textbf{Related Works.} Studying bi-level problems dates back to \cite{BraMcG73,von1952theory}. Since then, there has been a body of literature working on different forms of bi-level problems. A series of works have focused on replacing the lower-level problem with its optimality conditions as constraints for the upper-level one and thus reducing the problem to a single-level constrained optimization problems (see e.g., \cite{HaBrGi92,ShiLuZha05}). On the other hand, a number of iterative algorithms have been also proposed to directly tackle the bi-level problems (see e.g., \cite{CouWan16,CouWan15}).

More recently, due to the emerging applications of bi-level models, studying finite-time convergence analysis of iterative algorithms for solving bi-level optimization problems has gained renewed interest. Approximation algorithms with established such convergence analysis have been first proposed in \cite{ghadimi2018approximation} for convex/nonconvex and deterministic/stochastic settings and followed up by several works aimed to improve the complexity bounds under different structural assumptions or settings (see e.g., \cite{hong2023two, ji2021bilevel,chen2021closing,yang2021provably,huang2022efficiently,chen2023bilevel,chen2022decentralized, chen2023decentralized, yang2022decentralized}).

On the other hand, composition problems have been first studied in \cite{ermoliev1976methods} where the authors studied a penalized version of stochastic constraints in the objective function. There has been renewed interest in analyzing finite-time convergence analysis of SA-type algorithms over the past few years (see e.g., \cite{wang2017stochastic,WaLiFa17,yang2019multi}). However, all of these works obtained complexity bounds which are worse than their counterparts for classical stochastic optimization, which are single-level problems. The first optimal bound for two-level composition problems has been established in \cite{ghadimi2020single} and generalized to the multi-level case in \cite{balasubramanian2022stochastic}. Other works have also been developed to study nested composition problems under different assumptions, like non-smoothness (e.g., \cite{cui2020multicomposite,rusz20,liu2022solving}), mean-square smoothness (e.g.,~\cite{chen2021solving,zhang2021multilevel}) and dependence (e.g.,~\cite{hu2020biased,hu2021bias}).

Our main contributions in this paper consist of the following aspects.

\begin{itemize}
    \item We generalize the bi-level optimization problems by considering the case when the upper-level problem is a constrained optimization over a nested composition of $T$ functions, as defined in \eqref{eq:multilevel}. We propose \texttt{BiLiNASA}, a fully-online and batch-free stochastic approximation algorithm for solving \eqref{eq:multilevel}, based on a novel Neumann series approximation procedure for avoiding matrix inversions. To our knowledge, there has been no prior work considering the above mentioned combination of bi-level and nested compositional problems.
    \item We undertake a careful analysis of the oracle complexity of the proposed algorithm and show in Theorem~\ref{thm:mainresult} that \texttt{BiLiNASA} requires $\tilde{O}_T(1/\epsilon^{2})$ calls to the stochastic oracle to obtain an $\epsilon$-approximate stationary solution. Our proofs are based on \emph{non-trivial} adaptations of the existing analyses of bi-level and nested compositional problems;  the main difficulty arises in handling the different sources of biases in the stochastic gradients. 
    \item We model the problem of robust feature learning in predictive deep learning as a special case of the proposed bi-level framework. Through simulations we demonstrate that the learnt features using \texttt{BiLiNASA} are robust to a wide class of covariate shifts.
\end{itemize}

The rest of this paper is organized as follows. In Section \ref{sec:alg} we present our main algorithms, and in Section \ref{sec: main_results} we provide their convergence analyses. We then present simulation results in Section \ref{sec:experiments}. 

\begin{table}[t]
\centering
\begin{tabular}{|c|c|c|c|c|c|}
\hline
Method   & Samples & Mini-batch & Feasible Set & Matrix Inversion & $T$\\ \hline
\texttt{BSA}* \cite{ghadimi2018approximation} & $\tilde{O}(\epsilon^{-3})$ & No & $\R^{p}$  & No     & 1          \\ \hline
\texttt{TTSA}* \cite{hong2023two} & $\tilde{O}(\epsilon^{-2.5})$  & No & $\R^{p}$  & No   & 1            \\ \hline
\texttt{StocBiO} \cite{ji2021bilevel}  & $\tilde{O}(\epsilon^{-2})$ & Yes &$ \R^{p}$  & No    & 1           \\ \hline
\texttt{ALSET}* \cite{chen2021closing}    & $\tilde{O}(\epsilon^{-2})$ & No & $\R^{p}$  & No   & 1            \\ \hline
\texttt{STABLE} \cite{chen2022single}   & $O(\epsilon^{-2})$ & No & $\R^{p}$  & Yes        & 1      \\ \hline
\texttt{AmIGO} \cite{arbel2021amortized}    & $O(\epsilon^{-2})$ & Yes & $\R^{p}$ & No   & 1            \\ \hline
\texttt{SOBA}** \cite{dagreou2022framework} & $O(\epsilon^{-2})$ & No & $\R^{p}$ & No   & 1 \\ \hline
\texttt{MA-SOBA} \cite{chen2023optimal} & $O(\epsilon^{-2})$ & No & $X\subseteq\mathbb{R}^p$ & No   & 1 \\ \hline
\rowcolor{LightCyan} \texttt{BiLiNASA} & $\tilde{O}_T(\epsilon^{-2})$  & No & $X\subseteq\mathbb{R}^p$  & No        & $\geq 1$       \\ \hline
\end{tabular}
\caption{Sample complexity of gradient evaluations for obtaining an $\epsilon$-approximate first-order stationary solution of the upper-level problem in~\eqref{eq:multilevel}. Here, $\tilde{O}$ hides polylog factors and in the constrained case $X$ is assumed to be closed and convex.~\texttt{BiLiNASA} corresponds to Algorithm~\ref{algo: bilinasa}.  *\texttt{BSA}, \texttt{TTSA} and \texttt{ALSET} have an additional $\log$ factors (as stated here), arising due to an error in their analysis; see Appendix~\ref{sec: issue}. **\texttt{SOBA} requires an additional higher-order smoothness condition compared to other methods.}
\label{table: compare_bi-level}
\end{table}

\section{Preliminaries and Methodology}
\label{sec:alg}
In this section, we first provide the main assumptions and some technical results that will be later used in our analysis. We first need the following assumptions which are standard in both bi-level and compositional optimization literature.
\begin{assumption}\label{assump: fi_lips}
All functions $f_1,\dots,f_T$ and their derivatives are Lipschitz continuous with Lipschitz constants $L_{f_i}$ and $L_{\nabla f_i}$, respectively.
\end{assumption}

\begin{assumption}\label{assump: g}
    The lower-level function $g$ is twice continuously differentiable and $\mu_g$-strongly convex with respect to $y$ for any $x$. Its gradient $\nabla g$ and Hessian matrix $\nabla^2 g$ are Lipschitz continuous with Lipschitz constants $L_{\nabla g}$ and $L_{\nabla^2 g}$. 
\end{assumption}

We adopt some useful properties from bi-level optimization and compositional optimization literature. 
\begin{lemma}\label{lem: lip_of_psi}
    Suppose Assumption \ref{assump: fi_lips} holds. Then $\Psi(x,y)$ and $\nabla \Psi(x,y)$ are $L_{\Psi}$ and $L_{\nabla\Psi}$-Lipschitz continuous respectively with the constants given by
    \begin{align*}
        L_{\Psi} = \prod_{i=1}^{T}L_{f_i},\ L_{\nabla \Psi} = \sum_{j=1}^{T}\left[L_{\nabla f_j}\prod_{l=1}^{j-1}L_{f_l}\prod_{l=j+1}^{T}L_{f_l}^2\right]
    \end{align*}
\end{lemma}
\begin{proof}
    The expression of $L_{\Psi}$ is a direct result of Assumption \ref{assump: fi_lips}. The proof of the closed form expression of $L_{\nabla\Psi}$ can be found in Lemma 2.1 of \cite{balasubramanian2022stochastic}.
\end{proof}
\begin{lemma}[Lemma 2.2 in \cite{ghadimi2018approximation}]\label{lem: bi-level_smoothness}
    Suppose Assumptions \ref{assump: fi_lips} and \ref{assump: g} hold. Then the hypergradient $\nabla\Phi(x)$ takes the form
    \begin{align}
        \nabla \Phi(x) =  \nabla_x \Psi(x, y^*(x)) -  \nabla_{xy}^2 g(x, y^*(x))\cdot \nabla^2_{yy} g(x, y^*(x))^{-1} \nabla_y \Psi(x, y^*(x)).\label{hyper_grad}
    \end{align}
    Moreover, $y^*(x)$ and $\nabla\Phi(x)$ are $L_{y^*}$ and $L_{\nabla\Phi}$-Lipschitz continuous respectively with the constants given by $L_{y^*} = \frac{L_{\nabla g}}{\mu_g}$ and 
    \begin{align*}
        L_{\nabla \Phi} &= L_{\nabla_x \Psi} + \frac{(L_{\nabla_x \Psi} + L_{\nabla_y \Psi})L_{\nabla g} + L_{\Psi}^2L_{\nabla^2 g}}{\mu_g} + \frac{2L_{\Psi}L_{\nabla g}L_{\nabla^2 g}+L_{\nabla_y\Psi}L_{\nabla g}^2}{\mu_g^2} + \frac{L_{\Psi}L_{\nabla^2 g}L_{\nabla g}^2}{\mu_g^3},
    \end{align*}
    where $L_{\nabla_x \Psi}$ and $L_{\nabla_y \Psi}$ represent the Lipschitz constants of $\nabla_x \Psi$ and $\nabla_y \Psi$ respectively.
\end{lemma}


The difference is that now $\Psi$ is a compositional function and its gradient is:
\begin{equation}
    \begin{pmatrix}
        \nabla_x \Psi(x,y) \\
        \nabla_y \Psi(x,y)
    \end{pmatrix}=\nabla \Psi(x,y) = \nabla f_T(x, y)\nabla f_{T-1}(\tilde{x}_{T-1})\dots\nabla f_1(\tilde{x}_1), 
 \end{equation}
where $\nabla f_i$ denotes the transpose of the Jacobian matrix of $f_i$, and $\tilde{x}_i = f_{i+1}\circ \dots \circ f_T(x, y)$ for $1\leq i<T$.

Our methods generate different random sequences $\{d_k, x_k, u_k^{(i)}, y_k^{(i)}\}$, for which we define the following filtrations
\begin{align*}
    &\F_k = \sigma\left(\{u_0^{(1)},...,u_k^{(1)}, ...,u_0^{(T+1)},...,u_k^{(T+1)}, d_0, ..., d_k\}\right), \\
        &\G_{j}^{(t)} = \sigma\left(\bigcup_{l=0}^{j-1}\{x_l, y_l^{(0)}, y_l^{(1)}, ..., y_l^{(N)}\}\bigcup\{x_j, y_j^{(0)},y_j^{(1)},...,y_j^{(t)}\}\right).
\end{align*}

Moreover, we make the following standard assumptions on the outputs of the stochastic oracle which are used at each iteration of the algorithm.
\begin{assumption}\label{assump: stochastic_oracles}
Denote $u_{k}^{(T+1)} \equiv (x_k, y_k^{(N)})$. For each $i, k$ and $t$, the stochastic oracle outputs $F_{k+1}^{(i)} \in \mathbb{R}^{d_{i-1}},\ J_{k+1}^{(i)} \in \mathbb{R}^{d_i \times d_{i-1}},\ v_k^{(t)}\in \R^{q},\ J_g^{(k+1)}\in\R^{p\times q},$ and $H_{n}^{(k+1)} \in\R^{q\times q}$ such that for $i \in \{1,\ldots,T\}$,
\begin{enumerate}
\item The outputs are unbiased and have bounded variances:
\begin{align*}
    &\E\left[F_{k+1}^{(i)}|\F_k\right] = f_i(u_k^{(i+1)}),\ \E\left[\|F_{k+1}^{(i)}-f_i(u_k^{(i+1)})\|^2
        |\F_k\right]\leq \sigma_{F_i}^2,\\
        &\E\left[J_{k+1}^{(i)}|\F_k\right] = \nabla f_i(u_k^{(i+1)}),\ \E\left[\|J_{k+1}^{(i)} - \nabla f_i(u_k^{(i+1)})\|^2\right]\leq  \sigma_{J_i}^2,\ \E\left[\|J_{k+1}^{(i)}\|^2
        |\F_k\right]\leq \hat \sigma_{J_i}^2,\ \\
        &\E\left[v_k^{(t)}|\G_k^{(t)}\right] = \nabla_y g(x_{k},y_{k}^{(t)}),\ \E\left[\|v_k^{(t)}-\nabla_y g(x_{k},y_{k}^{(t)})\|^2|\G_k^{(t)}\right]\leq \sigma_v^2, \\
        &\E\left[J_g^{(k+1)}|\F_k\right] = \nabla_{xy}^2 g(x_{k},y_{k}^{(N)}),\ \E\left[\|J_g^{(k)}\|^2
        |\F_k\right]\leq \sigma_{J_g}^2, \\
        &\E\left[H_{n}^{(k+1)}|\F_k\right] = \nabla_{yy}^2g(x_k,y_k^{(N)}),\ \E\left[\|H_{n}^{(k+1)} - \nabla_{yy}^2g(x_k,y_k^{(N)})\|^2\right]\leq \sigma_{H_g}^2.
\end{align*}
\item Given $\F_k$, the outputs of the stochastic oracle at each level $i, F_{k+1}^{(i)}, J_{k+1}^{(i)}, J_g^{(k+1)}$ and $H_{n}^{(k+1)}$ are independent.
\item Given $\F_k$, the outputs of the stochastic oracle are independent between levels i.e., $\{F_{k+1}^{(i)}\}_{i=1,\ldots,T}$ are independent and so are $\{J_{k+1}^{(i)}\}_{i=1,\ldots,T}$.
\end{enumerate}
\end{assumption}

Recall from~\eqref{eq: hyper_grad} that the hypergradient of the bi-level problems involves a computationally demanding matrix inverse operation. Algorithms like \texttt{BSA}~\cite{ghadimi2018approximation}, \texttt{TTSA}~\cite{hong2023two}, and \texttt{ALSET}~\cite{chen2021closing} overcome this challenge by using a Neumann series approximation for the matrix inverse. However, we point out that these works suffer from an error in the analysis, which is detailed in Appendix~\ref{sec: issue}. In this work, we fix this issue and as a consequence we identify that the rates in the above works have an additional $\log$ factor, similar to our results. 

In this paper, we provide a different approach in which we estimate the product of the Hessian inverse and partial derivative in \eqref{hyper_grad} by taking a weighted average of stochastic Hessian-vector products as shown in Algorithm~\ref{algo: nhe}.
\begin{algorithm}[htp]
    \caption{\texttt{NHE}: Nested Hypergradient Estimation}
    \label{algo: nhe}
    \begin{algorithmic}[1]
        \REQUIRE{Points $x\in X, y\in \R^q, u^{(i)} \in \R^{d_i} (1\leq i\leq T)$, and parameters $\alpha > 0, M\in\N_+$.}
        \STATE{Call the stochastic oracle to compute stochastic  Jacobian matrix $J_g$ at $(x, y)$ and stochastic gradients\\ $J^{(i)}$ at $u^{(i+1)}$ for $i=1,...,T$ by denoting $u^{(T+1)} = (x, y)$.}
    \STATE{Set $r_0 = \begin{pmatrix}
        r_{0,x} \\
        r_{0,y} \\
    \end{pmatrix} = \prod_{i=1}^{T} J^{(T+1-i)} \in \R^{p+q},\ \bar{r}_{0,y}=0.$}
    \FOR{$n=1,2,...,M$}{
         \STATE{Compute the stochastic Hessian matrix $H_n$ and set
         \begin{equation}
           \bar{r}_{n,y} = (I - \alpha H_n )\bar{r}_{n-1,y} + \alpha r_{0,y}.\label{def_rbar}
             \end{equation}}
    }
    \ENDFOR
    \STATE{Set $r = r_{0,x} - J_g \bar{r}_{M,y}.$}
    \RETURN{$r$}
    \end{algorithmic}
\end{algorithm}
In this method, named as Nested Hypergradient Estimation, we first call the stochastic oracle to estimate the partial derivatives of the composition function $\Psi$ at a given point $(x,y)$ and then estimate the $(\nabla_{yy}^2 g)^{-1} \nabla_y \Psi$ in a loop. Finally, we can estimate the hypergradient by incorporating the stochastic second-order partial derivatives of the inner function $g$ as in Step 6 of this method.
We are now ready to present our method as Algorithm~\ref{algo: bilinasa}.


\begin{algorithm}[htp]
    \caption{\texttt{BiLiNASA}: Bi-level Linearized Nested Averaging Stochastic Approximation}
    \label{algo: bilinasa}
    \begin{algorithmic}[1]
        \REQUIRE{Initial points $x_0 \in X, y_0 \in \R^q, u_0^{(i)} \in \R^{d_i} (1\leq i\leq T)$, number of iterations for the outer loop $K$, number of iterations for the inner loop $N$, number of iterations and the averaging weight for the Hessian approximation subroutine $M$, and $\alpha$, and stepsize parameters $\beta>0, \ \ \tau_k \in (0,1], \ \ \gamma_k>0$.}
    	\FOR{$k=0,1,...,K-1$}{
    	    \STATE{$z_k = \argmin_{z\in X}\left\{\<d_k, z - x_k> + \frac{\beta}{2}\|z-x_k\|^2\right\}$.}
    		\STATE{$x_{k+1} = x_{k} +\tau_k (z_k - x_k)$.}
    	    \STATE{$y_{k+1}^{(0)} = y_{k}^{(N)} \text{ if } k>0, \text{ otherwise }  y_{k+1}^{(0)} = y^{(0)}.$}
    		\FOR{$t=0,1,..., N-1$}{
    		    \STATE{Call the stochastic oracle to compute $v_{k+1}^{(t)}$.}
                \STATE{$y_{k+1}^{(t+1)} = y_{k+1}^{(t)} -\gamma_{k+1} v_{k+1}^{(t)}.$}
            }
            \ENDFOR
    		\STATE{Run Algorithm \texttt{NHE}($x_k, y_k^{(N)}, u_k^{(i)}, \alpha, M$) and let the output be $w_{k+1}$.}
            \STATE{Set $d_{k+1} = (1 - \tau_{k})d_{k} + \tau_{k} w_{k+1}.$}
            \STATE{Compute stochastic functions values $F_{k+1}^{(i)}$ at $u_{k}^{(i+1)}$, and set}
            \STATE{$u_{k+1}^{(i)} = (1 - \tau_{k})u_{k}^{(i)} + \tau_{k} F_{k+1}^{(i)} + (J_{k+1}^{(i)})\T(u_{k+1}^{(i+1)} - u_{k}^{(i+1)})$ for $i = 1,...,T$.}
    	}
    	\ENDFOR
    	\RETURN{$x_K$}
    \end{algorithmic}
\end{algorithm}
We now make a few comments about the above algorithm. First, note that ignoring the inner loop, the NHE method, and assuming that $w_{k+1}$ is an unbiased gradient estimator for $\Phi$, Algorithm~\ref{algo: bilinasa} reduces to the one in \cite{balasubramanian2022stochastic} proposed for solving the nested compositonal optimization problems. However, due to the bi-level structure, we need another sequence ($y_k^{t}$) to update the decision variable for the lower problem. Second, while Algorithm NHE provides a biased estimate for the gradient of the upper problem, taking the weighted average of such estimates in Step 10 of the algorithm will help us to reduce the associated bias. Finally, using the linear approximation of the inner functions in the upper objective function will also reduce the noise associated with their function value estimates. 

\section{Main Result}\label{sec: main_results}
In this section, we will provide the convergence analysis of Algorithm~\ref{algo: bilinasa} under the standard assumptions that have been presented before. We first need to define our termination criterion for which the convergence of Algorithm~\ref{algo: bilinasa} is provided. 

Following \cite{ghadimi2020single, balasubramanian2022stochastic}, we define the measure of optimality $V_k$ as
\begin{equation}\label{eq: V_k}
    V_k = \|z_k - x_k\|^2 + \|d_k - \nabla\Phi(x_k)\|^2,
\end{equation}
which is commonly used as a measure of optimality in compositional optimization literature \cite{ghadimi2020single, balasubramanian2022stochastic}, since it provides an upper bound for gradient mapping \cite{ghadimi2016mini, nesterov2013gradient, nesterov2018lectures}. We now state the main convergence result of Algorithm~\ref{algo: bilinasa}.

\begin{theorem}\label{thm:mainresult}
    Suppose that Assumptions \ref{assump: fi_lips}, \ref{assump: g}, and \ref{assump: stochastic_oracles} hold. Moreover, assume that the stepsize sequences $\{\tau_j\}_{j=0}^{\infty}$ and $\{\gamma_j\}_{j=0}^{\infty}$, for any $k\geq 0$, satisfy
    \begin{equation}
        \begin{aligned}
            \tau_{k+1}\leq \tau_{k},\ 0< \gamma_{k+1}\leq \gamma_k\leq c_{\gamma}\tau_k\leq \tau_k\leq \tau_0<1 \text{ for some }c_{\gamma}>0,\label{assump: stepsize}
        \end{aligned}
    \end{equation}
    and $\alpha,\ N$ and $M$ also satisfy 
    \begin{align*}
        &0<\alpha < \min\left\{\frac{\mu_g}{\mu_g^2 + \sigma_{H_g}^2},\ \frac{1}{L_{\nabla g}} \right\}, N\geq \left(1 + \frac{1}{1-\tau_1}\right)\frac{\tau_k}{2\gamma_k\mu_g}, \delta_g := (1-\alpha\mu_g)^{M}\leq \frac{1}{2}.
    \end{align*}
Letting $R$ be chosen from $\{0,1,...,K\}$ with the probability mass function $P(R=k) = \frac{\tau_k}{\sum_{j=0}^{K}\tau_j}$, we have
    \[
        \E\left[V_R\right] = O_T\left(\frac{\sum_{k=0}^{K}\tau_k^2 + 1}{\sum_{k=0}^{K}\tau_k} + \delta_g^2\right).
    \]
\end{theorem}

\noindent \textbf{Sample complexity.} If we pick $M = \Theta(\log K)$, and $$\tau_k =\Theta \left(\frac{1}{\sqrt{K}}\right), \gamma_k = \Theta\left(\frac{1}{\sqrt{K}}\right), \alpha =  \frac{1}{2}\cdot\min\left\{\frac{\mu_g}{\mu_g^2 + \sigma_{H_g}^2},\ \frac{1}{L_{\nabla g}}\right\},$$
then the oracle complexity of stochastic gradients, Jacobian-vector products and Hessian-vector products will be $O_T(KM) = O_T(\epsilon^{-2}\log\frac{1}{\epsilon})$ to guarantee $\E\left[V_R\right] = O(\epsilon)$, which matches the rate in \cite{balasubramanian2022stochastic}\footnote{The work of \cite{balasubramanian2022stochastic} considers $\E\left[\sqrt{V_R}\right] = O(\epsilon)$ and obtains $O_T(\epsilon^{-4})$.} if we ignore the log factor. 

\noindent \textbf{Difficulty in removing the $\log$ factor.}
A recent work \cite{dagreou2022framework} introduces \texttt{SOBA}, which removes the $\log$ factor in the sample complexity of stochastic bi-level optimization (i.e., $T=1$ in \eqref{eq:multilevel}) comparing to previous results \cite{chen2021closing}. However, it is unclear if a similar method can be incorporated into our algorithm and its convergence analysis. This is due to the fact that the upper-level function of our problem \eqref{eq:multilevel} has a nested structure, which makes the unbiased estimate of the upper-level gradient unavailable. Note that this bias does not exist when $T=1$ and it causes additional major challenges in directly incorporating  the framework of~\cite{dagreou2022framework} to solve \eqref{eq:multilevel}.

\subsection{Proof of Theorem~\ref{thm:mainresult}}
Before providing the proof of Theorem~\ref{thm:mainresult}, we first need a few simple technical results.

\begin{lemma}\label{lem: ineq_of_seqs}
For a sequence $\{\tau_k\}_{k=0}^{\infty} \in (0,1)$ define $\Gamma_0 = 1,\ \Gamma_{j+1} = \prod_{i=0}^{j}(1-\tau_i)$. Then, for any $j\geq 0$, we have
    \begin{align*}
         \sum_{j=0}^{k}\frac{\tau_j\Gamma_{k+1}}{\Gamma_{j+1}}< 1,\quad \sum_{j=k}^{K}\frac{\tau_j\Gamma_j}{\Gamma_k}< 1\quad \text{for any } 0\leq k\leq K,
    \end{align*}
where $K \ge 0$.
\end{lemma}
\begin{proof}
    Note that by definition of $\Gamma_j$, we have
    \[
        \sum_{j=0}^{k}\frac{\tau_j\Gamma_{k+1}}{\Gamma_{j+1}} = \sum_{j=0}^{k}\tau_j\prod_{i=j+1}^{k}(1-\tau_i) = \sum_{j=0}^{k}\left(\prod_{i=j+1}^{k}(1-\tau_i) - \prod_{i=j}^{k}(1-\tau_i)\right) = 1 - \prod_{i=0}^{k}(1-\tau_i),
    \]
    where we set $\prod_{i=k+1}^{k}(1-\tau_i) = 1$ in the first equality. Similarly, we have
    \[
        \sum_{j=k}^{K}\frac{\tau_j\Gamma_j}{\Gamma_k} = \sum_{j=k}^{K}\tau_j \prod_{i=k}^{j-1}(1-\tau_i) = \sum_{j=k}^{K}\left(\prod_{i=k}^{j-1}(1-\tau_i) - \prod_{i=k}^{j}(1-\tau_i)\right) = 1 - \prod_{i=k}^{K}(1-\tau_i).
    \]
    Noting that $\prod_{i=0}^{k}(1-\tau_i)$ and $\prod_{i=k}^{K}(1-\tau_i)$ are both positive and less than $1$, the proof is complete.
\end{proof}

The next lemma provides a tight estimation of weighted sums under certain conditions.
\begin{lemma}\label{lem: decrease_lemma}
    Suppose we are given five sequences $\{a_n\}_{n=1}^{\infty}, \{b_n\}_{n=1}^{\infty}, \{c_n\}_{n=0}^{\infty},$ \\ $\{\tau_n\}_{n=0}^{\infty},$ and $\{\delta_n\}_{n=1}^{\infty}$ satisfying
    \begin{align*}
        &a_{k+1}\leq \delta_k a_k + b_{k},\ \Gamma_0 = 1,\ \Gamma_{k} = \prod_{l=1}^{k}\delta_l,\ \sum_{k=i}^{K}\tau_k\Gamma_k\leq c_i\Gamma_i,\\
        &a_k\geq 0,\ b_k\geq 0,\ c_i\geq 0,\ \tau_i \geq 0,\ 0\leq \delta_k < 1,
    \end{align*}
    for all $k=1,2,...$ and $i=0,1,..$. Then, for any $K > 0$, we have
    \begin{equation}
        a_{k+1}\leq a_1\Gamma_{k} + \sum_{i=1}^{k}\frac{b_i\Gamma_k}{\Gamma_i},\qquad \sum_{k=0}^{K}\tau_ka_{k+1}\leq c_0a_1 + \sum_{i=1}^{K}c_i b_i.\label{sum_ak}
    \end{equation}
\end{lemma}
\begin{proof}
    First note that by the definition of $\Gamma_k$, for any $k\geq 1$, we have
    \[
        \frac{a_{k+1}}{\Gamma_{k}} \leq \frac{a_k}{\Gamma_{k-1}} + \frac{b_{k}}{\Gamma_{k}},
    \]
which implies that
    \[
        a_{k+1}\leq 
        a_1\Gamma_{k} + \sum_{i=1}^{k}\frac{b_i\Gamma_k}{\Gamma_i}.
    \]
    By denoting $\sum_{i=1}^{0}\frac{b_i\Gamma_k}{\Gamma_i} = 0$, we obtain
    \begin{align*}
        &\sum_{k=0}^{K}\tau_ka_{k+1}\leq \sum_{k=0}^{K}a_1\tau_k\Gamma_{k} + \sum_{k=0}^{K}\sum_{i=1}^{k}\frac{b_i\tau_k\Gamma_k}{\Gamma_i} = \sum_{k=0}^{K}a_1\tau_k\Gamma_{k} + \sum_{i=1}^{K}\sum_{k=i}^{K}\frac{b_i\tau_k\Gamma_k}{\Gamma_i} \\
        \leq &a_1\sum_{k=0}^{K}\tau_k\Gamma_{k} + \sum_{i=1}^{K}\frac{b_i}{\Gamma_i}\sum_{k=i}^{K}\tau_k\Gamma_k \le c_0a_1 + \sum_{i=1}^{K}c_i b_i,
    \end{align*}
    where the last inequality holds due to the fact that $\sum_{k=i}^{K}\tau_k\Gamma_k\leq c_i\Gamma_i$.
\end{proof}
{\bf Remark:} By choosing $\delta_k \equiv \delta \in (0, 1)$ and $c_i = \frac{\tau_i}{1-\delta}$ with ${\tau_i}$ being a decreasing sequence, 
we know
\[
    \sum_{k=i}^K \tau_k\Gamma_k = \sum_{k=i}^K\tau_k\delta^k\leq \tau_i\delta^i\cdot \left(\sum_{k=i}^{K}\delta^{k-i}\right)\leq c_i\Gamma_i.
\]
Furthermore, if $a_k, b_k$ and $c_k$ are chosen such that other conditions are satisfied,  we know by \eqref{sum_ak} that
\begin{align}
    \sum_{k=0}^{K}\tau_ka_{k+1} &\leq \frac{\tau_0 a_1}{1-\delta} + \sum_{i=1}^{K}\frac{\tau_i b_i}{1-\delta},\notag \\ 
    a_{k+1}^2&\leq \left(a_1\delta^k + \sum_{i=1}^{k}b_i\delta^{k-i}\right)^2
    \leq \left(\delta^k + \sum_{i=1}^{k}\delta^{k-i}\right)\left(a_1^2\delta^k + \sum_{i=1}^{k}b_i^2\delta^{k-i}\right)\nonumber \\
    &<\frac{1}{1-\delta}\left(a_1^2\delta^{k} + \sum_{i=1}^{k}b_i^2\delta^{k-i}\right),\nonumber
\end{align}
which, also gives
\begin{equation}\label{ineq: decrease_coro2}
    \sum_{k=0}^{K}\tau_ka_{k+1}^2\leq \frac{\tau_0a_1^2}{(1-\delta)^2} + \sum_{i=1}^{K}\frac{\tau_ib_i^2}{(1-\delta)^2}.
\end{equation}
We also need the following standard result in stochastic optimization.
\begin{lemma}[Lemma 10 in \cite{qu2017harnessing}]\label{lem: gd_decrease}
    Suppose $f(x)$ is $\mu$-strongly convex and $L-smooth$. For any $x$ and $\gamma<\frac{2}{\mu + L}$, define $x^+ = x - \gamma\nabla f(x),\ x^*=\argmin f(x)$. Then we have $\|x^+ - x^*\| \leq (1-\gamma\mu)\|x-x^*\|$.
\end{lemma}

The next result establishes the smoothness of the optimal value of the subproblem solved at Step 2 of Algorithm~\ref{algo: bilinasa}.
\begin{lemma}[\cite{ghadimi2020single}]\label{lem: eta_smooth}
    Define $\eta(x, d, \beta) = \min_{y\in X}\left\{\<z,y-x> + \frac{\beta}{2}\|y-x\|^2\right\}.$
    Then $\nabla \eta$ is $L_{\nabla\eta}$-Lipschitz continuous with the constant given by \[
    L_{\nabla\eta} = 2 \sqrt{\left(1+\beta\right)^2 + \left(1 + \frac{1}{2\beta}\right)^2}.\]
\end{lemma}

To establish the convergence of Algorithm~\ref{algo: bilinasa}, we first analyze the convergence of its inner loop updates in the next lemma.
\begin{lemma}\label{lem: inner_analysis}
Let the sequence $y_k^{(t)}$ be generated by Algorithm \ref{algo: bilinasa} and denote $y_k^* := y^*(x_k)$. Suppose that Assumption \ref{assump: g} holds and stepsizes satisfy \eqref{assump: stepsize}. If $\gamma_k<\frac{2}{\mu + L}$, for any $k\geq 1$, the inner loop updates, for any $0\leq t\leq N$, satisfy
    \begin{align}
            &\E\left[\|y_{k+1}^{(N)} - y_k^{(N)}\|^2\right]\label{ineq: inner_part1_2} \\
            &\leq N^2\gamma_{k+1}^2\sigma_v^2
            + \min\left\{N\gamma_{k+1},\ \frac{1}{\mu_g}\right\}NL_{\nabla g}\gamma_{k+1}\E\left[\|y_{k+1}^{(0)} - y_{k+1}^*\|^2\right] + \frac{N^3L_{\nabla g}\gamma_{k+1}^4\sigma_v^2}{2}.\notag
    \end{align}
    Moreover, if $N\geq (1 + \frac{1}{1-\tau_1})\frac{\tau_k}{2\gamma_k\mu_g}$, we have
    \begin{align}
        \E\left[\|y_{k}^{(0)} - y_k^*\|^2\right] &\leq \frac{\Gamma_k}{\Gamma_1}\E\left[\|y_{1}^{(0)} - y_1^*\|^2\right] + \frac{2\sigma_v^2c_{\gamma}}{\mu_g} + 2L_{y^*}^2\max_{1\leq i\leq k}\E\left[\|x_k-z_k\|^2\right]\label{ineq: inner_part2_1} \\
            \sum_{k=1}^{K}\tau_{k}\E\left[\|y_k^{(0)} - y_k^*\|^2\right] &\le \E\left[\|y_1^{(0)} - y_1^*\|^2\right] + 2 \sum_{k=1}^{K}\tau_k \Big(N\sigma_v^2c_{\gamma}^2\tau_k+ L_{y^*}^2\E\left[\|x_k-z_k\|^2\right] \Big). \label{ineq: inner_part2_2}
    \end{align}
\end{lemma}
\begin{proof}
By definition of the $\sigma$-algebra $\G_k^{(t)}$, the update rule of $y_k^{(t)}$ in Step 7 of Algorithm \ref{algo: bilinasa}, and under Assumption~\ref{assump: stochastic_oracles}, we have
    \begin{align*}
        &\E\left[\|y_k^{(t+1)} - y_k^*\|^2|\G_k^{(t)}\right] \\
        = &\E\left[\|y_k^{(t)} - \gamma_k\nabla_yg(x_k, y_k^{(t)}) - y_k^*+ \gamma_k(\nabla_yg(x_k,y_k^{(t)}) - v_k^{(t)})\|^2|\G_k^{(t)}\right] \\
        = &\|y_k^{(t)} - \gamma_k\nabla_yg(x_k, y_k^{(t)}) - y_k^*\|^2 + \gamma_k^2\E\left[\|\nabla_yg(x_k,y_k^{(t)}) - v_k^{(t)}\|^2|\G_k^{(t)}\right] \\
        \leq & (1-\gamma_k\mu_g)^2\|y_k^{(t)} - y_k^*\|^2 + \gamma_k^2\sigma_v^2,
    \end{align*}
where the inequality follows from Lemma \ref{lem: gd_decrease}. Taking expectation from both sides of the above inequality, we obtain 
    $\E\left[\|y_k^{(t+1)} - y_k^*\|^2\right]\leq (1-\gamma_k\mu_g)^2\E\left[\|y_k^{(t)} - y_k^*\|^2\right] + \gamma_k^2\sigma_v^2$, implying that
        \begin{align}
            \E\left[\|y_k^{(t)} - y_k^*\|^2\right]&\leq (1-\gamma_k\mu_g)^{2t}\E\left[\|y_k^{(0)} - y_k^*\|^2\right] + \gamma_k^2\sigma_v^2\sum_{i=0}^{t-1}(1-\gamma_k\mu_g)^{2i} \notag\\ 
            &\leq (1-\gamma_k\mu_g)^{2t}\E\left[\|y_k^{(0)} - y_k^*\|^2\right] + \min\left\{t\gamma_k^2\sigma_v^2, \frac{\gamma_k\sigma_v^2}{\mu_g}\right\},\label{ineq: inner_part1_1}
        \end{align}
    where the second inequality follows from the fact that 
    \begin{equation}\label{sum_delta}
        \sum_{i=0}^{t-1}(1-\gamma_k\mu_g)^{2i} \leq \sum_{i=0}^{t-1} (1-\gamma_k\mu_g)^i \leq \frac{1}{\gamma_k\mu_g}.
    \end{equation}
    Now, observe that
    \begin{align*}
        \|y_{k+1}^{(N)} - y_k^{(N)}\|^2 = \|y_{k+1}^{(N)} - y_{k+1}^{(0)}\|^2 = \Big\|\sum_{t=0}^{N-1}(y_{k+1}^{(t+1)} - y_{k+1}^{(t)})\Big\|^2 = \gamma_{k+1}^2\Big\|\sum_{t=0}^{N-1}v_{k+1}^{(t)}\Big\|^2 
        \leq N\gamma_{k+1}^2\sum_{t=0}^{N-1}\Big\|v_{k+1}^{(t)}\Big\|^2,
    \end{align*}
    which together with the fact that $\nabla_yg(x_{k+1},y_{k+1}^*) = 0$ and under Assumption~\ref{assump: stochastic_oracles},
    imply that
    \begin{align*}
        &\E\left[\|y_{k+1}^{(N)} - y_k^{(N)}\|^2|\G_{k+1}^{(N-1)}\right] \\
        \leq &N\gamma_{k+1}^2\sum_{t=0}^{N-1}\left\{\E\left[\|v_{k+1}^{(t)}-\nabla_y g(x_{k+1},y_{k+1}^{(t)})\|^2|\G_{k+1}^{(t)}\right] + \|\nabla_y g(x_{k+1}, y_{k+1}^{(t)})\|^2 \right\} \\
        \leq & N^2\gamma_{k+1}^2\sigma_v^2 + N\gamma_{k+1}^2\sum_{t=0}^{N-1}\|\nabla_yg(x_{k+1},y_{k+1}^{(t)}) - \nabla_y g(x_{k+1},y_{k+1}^*)\|^2\\
        \leq &N^2\gamma_{k+1}^2\sigma_v^2 + NL^2_{\nabla g}\gamma_{k+1}^2\sum_{t=0}^{N-1}\|y_{k+1}^{(t)} - y_{k+1}^*\|^2.
    \end{align*}
    Taking expectation on both sides and noting \eqref{ineq: inner_part1_1}, we obtain
    \begin{align*}
        &\E\left[\|y_{k+1}^{(N)} - y_k^{(N)}\|^2\right]\leq N^2\gamma_{k+1}^2\sigma_v^2 + NL^2_{\nabla g}\gamma_{k+1}^2\sum_{t=0}^{N-1}\left\{(1-\gamma_{k+1}\mu_g)^{2t}\E\left[\|y_{k+1}^{(0)} - y_{k+1}^*\|^2\right] + t\gamma_{k+1}^2\sigma_v^2\right\}
    \end{align*}
which together with \eqref{sum_delta}, imply \eqref{ineq: inner_part1_2}.

To show \eqref{ineq: inner_part2_2}, we first consider the decrease of $\E\left[\|y_k^{(0)} - y_k^*\|^2\right]$. Noting \eqref{ineq: inner_part1_1}, the fact that $y^*(x)$ is $L_{y^*}$-Lipschitz continuous due to Lemma \ref{lem: bi-level_smoothness}, and Step 3 of Algorithm~\ref{algo: bilinasa}, we have
    \begin{equation}\label{ineq: y_k0_decrease}
        \begin{aligned}
            &\E\left[\|y_{k+1}^{(0)} - y_{k+1}^*\|^2\right] = \E\left[\|y_{k}^{(N)} - y_{k}^* + y_{k}^* - y_{k+1}^*\|^2\right] \\
            \leq &\E\left[(1+\tau_k)\|y_{k}^{(N)} - y_{k}^*\|^2 + \left(1+\frac{1}{\tau_k}\right)\|y_k^* - y_{k+1}^*\|^2\right] \\
            \leq &(1+\tau_k) \left[(1-\gamma_k\mu_g)^{2N}\E\left[\|y_{k}^{(0)} - y_{k}^*\|^2\right] + N\gamma_{k}^2\sigma_v^2 \right]+ (\tau_k+\tau_k^2)L_{y^*}^2\E\left[\|x_{k}-z_{k}\|^2\right].
        \end{aligned}
    \end{equation}
   If we set $N\geq (1 + \frac{1}{1-\tau_1})\frac{\tau_k}{2\gamma_k\mu_g}$, we have
    \[
        (1+\tau_k)(1-\gamma_k\mu_g)^{2N} = e^{\log(1+\tau_k) + 2N\log(1-\mu_g\gamma_k)}\leq e^{\tau_k - 2N\mu_g\gamma_k}\leq e^{\frac{-\tau_k}{1-\tau_k}}\leq 1-\tau_k,
    \]
    where the first and third inequality follow from the fact that $\frac{x}{1+x}\leq \log(1+x)\leq x$ for any $x>-1$, and the second inequality follows from $N \geq (1 + \frac{1}{1-\tau_1})\frac{\tau_k}{2\gamma_k\mu_g}\geq (1 + \frac{1}{1-\tau_k})\frac{\tau_k}{2\gamma_k\mu_g}$. The above observation together with \eqref{ineq: y_k0_decrease} imply that
    \begin{align}
        \E\left[\|y_{k+1}^{(0)} - y_{k+1}^*\|^2\right] - \E\left[\|y_{k}^{(0)} - y_{k}^*\|^2\right]\leq -\tau_k\E\left[\|y_{k}^{(0)} - y_{k}^*\|^2\right] + 2N\gamma_{k}^2\sigma_v^2 + 2L_{y^*}^2\tau_{k}\E\left[\|x_{k}-z_{k}\|^2\right].\label{ineq: delta_inner_multi_original}
    \end{align}
    Taking summation on both sides and using \ref{assump: stepsize}, we obtain \eqref{ineq: inner_part2_2}. To prove \eqref{ineq: inner_part2_1}, we
    first notice that in \eqref{ineq: y_k0_decrease} we use another upper bound in \eqref{ineq: inner_part1_1} and follow the same process of proving the above inequality we may get:
    \begin{align*}
        &\E\left[\|y_{k+1}^{(0)} - y_{k+1}^*\|^2\right] \\
        \leq &(1+\tau_k)\delta_{k}^N\E\left[\|y_{k}^{(0)} - y_{k}^*\|^2\right] + (1+\tau_k)\frac{\gamma_k\sigma_v^2}{\mu_g} + (\tau_{k} + \tau_{k}^2)L_{y^*}^2\E\left[\|x_{k}-z_{k}\|^2\right] \\
        \leq &(1-\tau_k)\E\left[\|y_{k}^{(0)} - y_{k}^*\|^2\right] + \frac{2\gamma_k\sigma_v^2}{\mu_g} + 2L_{y^*}^2\tau_k\E\left[\|x_k-z_k\|^2\right].
    \end{align*}
    In the view of Lemma \ref{lem: decrease_lemma}, we have
    \begin{align*}
        &\frac{\E\left[\|y_{k+1}^{(0)} - y_{k+1}^*\|^2\right]}{\Gamma_{k+1}}\\
        \leq &\frac{\E\left[\|y_{1}^{(0)} - y_{1}^*\|^2\right]}{\Gamma_1} + \frac{2\sigma_v^2}{\mu_g}\sum_{i=1}^{k}\frac{\gamma_i}{\Gamma_{i+1}} + 2L_{y^*}^2\sum_{i=1}^{k}\frac{\tau_i}{\Gamma_{i+1}}\E\left[\|x_i-z_i\|^2\right] \\
            \leq &\frac{\E\left[\|y_{1}^{(0)} - y_{1}^*\|^2\right]}{\Gamma_{1}} + \frac{2\sigma_v^2c_{\gamma}}{\mu_g\Gamma_{k+1}} + \frac{2L_{y^*}^2}{\Gamma_{k+1}}\max_{1\leq i\leq k}\E\left[\|x_i-z_i\|^2\right].
    \end{align*}
    The second inequality uses Assumption \ref{assump: stepsize} and Lemma \ref{lem: ineq_of_seqs}. Multiplying $\Gamma_{k+1}$ on both sides completes the proof.
\end{proof}
{\bf Remark:} If we pick $\gamma_k = \Theta(\tau_k)$ then it suffices to pick $N=1$, which is independent of the iteration number $k$. This suggests using the same timescale for both loops, which matches the result in \cite{chen2021closing}.

Now, note that \eqref{def_rbar} can be written as
\begin{equation}\label{eq: sgdlike_update}
    \bar{r}_{n,y}^{(k+1)} = \bar{r}_{n-1,y}^{(k+1)} - \alpha H_n^{(k+1)} \bar{r}_{n-1,y}^{(k+1)} + \alpha r_{0,y}^{(k+1)},\quad n=1,2,...,M,
\end{equation}
which is essentially a SGD-like update. If we fix $k$ and define 
\[
    x_n := \bar{r}_{n,y}^{(k+1)},\ A_n := H_n^{(k+1)},\ A := \nabla_{yy}^2g(x_k,y_k^{(N)}),\ b_0 := r_{0,y}^{(k+1)},
\]
the above equation becomes $
    x_n = x_{n-1} - \alpha A_n x_{n-1} + \alpha b_0$, with $A_n$ being the unbiased estimator of $A$ and $b_0$ being a biased estimator of $\nabla_y\Psi(x_k,y_k^{(N)})$, since under Assumption~\ref{assump: stochastic_oracles}, we have
\begin{align*}
    \E\left[A_n\right] &= A,\ \E\left[\|A_n - A\|^2\right]\leq \sigma_{H_g}^2,\\
    \E\left[b_0\right] &= \E\left[r_{0,y}^{(k+1)}\right] = \nabla_y f_T(x_k, y_k^{(N)})\prod_{i=2}^{T} \nabla f_{T+1-i}(u_k^{(T+2-i)}).
\end{align*}
Thus, \eqref{eq: sgdlike_update} can be viewed as a $M$-step SGD applied to the following quadratic optimization problem $\min_x \{\frac{1}{2}x\T Ax - b_0\T x\}$, where $b_0$ is obtained from the stochastic oracle ahead of the first update $x_1$ and is fixed during the $M$-step updates. Using standard analysis of SGD, we can bound the variance and the second moment of $x_k$ via the following lemma:

\begin{lemma}\label{lem: SGD_var}
    Suppose that we are given a vector $b_0\in\R^{q}$ and a symmetric positive definite matrix $A\in\R^{q\times q}$  satisfying $\mu I\preceq A\preceq L I$ for $0<\mu\leq L$. Moreover, a sequence $\{x_k\}_{k=0}^{\infty}$ is defined as
    \[
        x_k = (I - \alpha A_{k})x_{k-1} + \alpha b_0,\ x_0 = 0,
    \]
    where $A_k$ satisfies $
        \E\left[A_k\right] = A$, $ \E\left[\|A_k-A\|^2\right]\leq \sigma^2$,     and $A_1,..., A_k$ are independent, $A_i$ and $x_{i-1}$ are also independent. If $\alpha$ satisfies
    \[
        0<\alpha <\min\left\{\frac{\mu}{\mu^2 + \sigma^2},\ \frac{1}{L}\right\},
    \]
    we have $ \E\left[\|x_k - \E\left[x_k\right]\|^2\right]< \frac{\|b_0\|^2}{\mu^2}$ and $ \E\left[\|x_k\|^2\right]< \frac{2\|b_0\|^2}{\mu^2}.$
\end{lemma}
\begin{proof}
For each $x_k$, we have $x_k = (I - \alpha A_{k})x_{k-1} + \alpha b_0,$ and hence $\E\left[x_k\right] = (I - \alpha A)\E\left[x_{k-1}\right] + \alpha b_0$, which give the closed form of $x_k$ and $\E\left[x_k\right]$ as 
\begin{align*}
    x_{k} &= \alpha\sum_{p=0}^{k-1}\prod_{i=1}^{p}(I - \alpha A_{k+1-i})b_0 + \prod_{i=1}^k(I - \alpha A_{k+1-i})x_0, \\
    \E\left[x_k\right]&=\alpha\left[\sum_{p=0}^{k-1}(I-\alpha A)^p\right]b_0 + (I - \alpha A)^k \E\left[x_0\right].
\end{align*}
Together with $x_0 = 0$, the above implies that
\begin{equation}\label{ineq: Ex}
    \begin{aligned}
        \|\E\left[x_k\right]\| &= \|(I - (I-\alpha A)^{k})A^{-1}b_0\|\leq \frac{\|b_0\|}{\mu}.
    \end{aligned}
\end{equation}
Hence we know
\begin{equation}\label{eq: x_recursion}
    \begin{aligned}
        &\|x_k - \E\left[x_k\right]\|^2 = \|(I - \alpha A)(x_{k-1} - \E\left[x_{k-1}\right]) + \alpha (A - A_k)x_{k-1}\|^2 \\
    =&\|(I - \alpha A)(x_{k-1} - \E\left[x_{k-1}\right])\|^2 + \alpha^2 \|(A - A_{k})x_{k-1}\|^2 \\
    + &2\alpha \<(I - \alpha A)(x_{k-1} - \E\left[x_{k-1}\right]), (A - A_{k})x_{k-1}>.
    \end{aligned}
\end{equation}
Taking expectation on both sides, we know:
\begin{equation}\label{ineq: x_var}
    \begin{aligned}
        &\E\left[\|x_k - \E\left[x_k\right]\|^2\right] \\
    = &\E\left[\|(I - \alpha A)(x_{k-1} - \E\left[x_{k-1}\right])\|^2\right] + \alpha^2\E\left[\|(A - A_{k})x_{k-1}\|^2\right]\\
    \leq &(1-\alpha\mu)^2\E\left[\|x_{k-1} - \E\left[x_{k-1}\right]\|^2\right] + \alpha^2\sigma^2(\E\left[\|x_{k-1}-\E\left[x_{k-1}\right]\|^2\right] + \|\E\left[x_{k-1}\right]\|^2) \\
    \leq & (1-\alpha\mu)\E\left[\|x_{k-1} - \E\left[x_{k-1}\right]\|^2\right] + \frac{\alpha^2\sigma^2\|b_0\|^2}{\mu^2} \\
    \leq &(1-\alpha\mu)^k\E\left[\|x_0 - \E\left[x_0\right]\|^2\right] + \frac{\alpha^2\sigma^2\|b_0\|^2}{\mu^2}\cdot \left(\sum_{i=0}^{k-1}(1-\alpha\mu)^i\right) < \frac{\alpha\sigma^2\|b_0\|^2}{\mu^3}\leq \frac{\|b_0\|^2}{\mu^2}.
    \end{aligned}
\end{equation}
The second inequality is due to a direct result of $\alpha\leq \frac{\mu}{\mu^2+\sigma^2}$:
\[
    (1-\alpha\mu)^2 + \alpha^2\sigma^2\leq 1 - \alpha\mu,
\]
and the fifth inequality uses $\alpha \leq \frac{\mu}{\mu^2 + \sigma^2}\leq \frac{\mu}{\sigma^2}$. For the second moment we have
\begin{equation}\label{ineq: x_2m}
    \E\left[\|x_k\|^2\right] = \E\left[\|x_k - \E\left[x_k\right]\|^2\right] + \|\E\left[x_k\right]\|^2 < \frac{2\|b_0\|^2}{\mu^2}.
\end{equation}
\eqref{ineq: x_var} and \eqref{ineq: x_2m} completes the proof.

\end{proof}
A direct result of Lemma \ref{lem: SGD_var} is the following lemma, which indicates $\bar{r}_{M,y}^{(k+1)}$ in Algorithm \ref{algo: nhe} has bounded variance and bounded second moment, and so does $r^{(k+1)}$.
\begin{lemma}\label{lem: bdd_var}
Suppose that Assumptions \ref{assump: fi_lips}, \ref{assump: stochastic_oracles}, and \ref{assump: g} hold. Define positive constants $\hat \sigma_r, \sigma_{\bar{r}}, \sigma_w$ as
\begin{align*}
    \hat \sigma_r^2 = \prod_{l=1}^{T}\hat \sigma_{J_l}^2,\quad \sigma_{\bar{r}}^2 =\frac{2\hat\sigma_r^2}{\mu_g^2},\quad \sigma_w^2 = (\hat \sigma_r + \sigma_{J_g}\sigma_{\bar{r}})^2.
\end{align*}
    In Algorithm \ref{algo: nhe}, if $\alpha$ satisfy
    \begin{equation}\label{ineq: epsilon_and_M}
        0<\alpha< \min\left\{\frac{\mu_g}{\mu_g^2 + \sigma_{H_g}^2},\ \frac{1}{L_{\nabla g}}\right\},
    \end{equation}
    the output $r$ satisfies $
        \E\left[\|r - \E\left[r\right]\|^2\right]\leq  \E\left[\|r\|^2\right]\leq \sigma_w^2$.
\end{lemma}
\begin{proof}
    The variance and second moment of $r_0$ in Step 2 of Algorithm \ref{algo: nhe} are bounded under Assumption~\ref{assump: stochastic_oracles} since
    \begin{equation}\label{ineq: r_0_bdd_var_multi}
        \begin{aligned}
            \E\left[\|r_0 - \E\left[r_0\right]\|^2\right]\leq \E\left[\|r_0\|^2\right]= \E\left[\left\|\prod_{l=1}^{T}J^{(l)}\right\|^2\right]\leq \prod_{l=1}^{T}\hat \sigma_{J_l}^2 =\hat \sigma_r^2.
        \end{aligned}
    \end{equation}

    By \eqref{def_rbar} and in the view of Lemma \ref{lem: SGD_var}, we have $\E\left[\|\bar{r}_{M,y}\|^2\right] \leq \frac{2\|r_{0,y}\|^2}{\mu_g^2}$. Taking expectation on both sides of the above inequality, noting that $\bar{r}_{0,y} = 0$ and \eqref{ineq: r_0_bdd_var_multi}, we have
    \begin{equation}\label{ineq: r_bar_2m}
        \E\left[\|\bar{r}_{M,y}\|^2\right] < \frac{2\hat \sigma_r^2}{\mu_g^2} = \sigma_{\bar{r}}^2.
    \end{equation}
    Then for the second moment of $r$, we have:
    \begin{align*}
        \E\left[\|r\|^2\right] = \E\left[\|r_{0,x} - J_g\cdot \bar{r}_{M,y}\|^2\right]\leq\E\left[(\|r_{0,x}\|+\|J_g\cdot \bar{r}_{M,y}\|)^2\right]\leq (\hat \sigma_r + \sigma_{J_g}\sigma_{\bar{r}})^2,
    \end{align*}
    implying $
        \E\left[\|r - \E\left[r\right]\|^2\right]\leq  \E\left[\|r\|^2\right]\leq \sigma_w^2$.
\end{proof}

We also need the following result about the output of Algorithm \ref{algo: nhe}.

\begin{lemma}\label{lem: hypergrad_exp}
    Suppose that Assumptions \ref{assump: fi_lips}, \ref{assump: g}, \ref{assump: stochastic_oracles} hold, and $\alpha$ satisfies \eqref{ineq: epsilon_and_M}. Then we have $\E\left[r\right] = \E\left[r_{0,x}\right] - \nabla_{xy}^2 g(x, y)\left[\nabla_y^2g(x,y)\right]^{-1}\E\left[r_{0,y}\right] + \mathcal{E}$, where $r$ is the output of Algorithm \ref{algo: nhe} and
    \begin{align}
        \mathcal{E} = \nabla_{xy}^2 g(x, y)\left[I - \alpha \nabla_y^2 g(x, y)\right]^M\bar{r}_{*,y},\ \bar{r}_{*,y} = \left[\nabla_y^2g(x, y)\right]^{-1} \E\left[r_{0,y}\right],\label{def_er}
    \end{align}
    and 
    \[
        \|\mathcal{E}\|\leq (1-\alpha \mu_g)^M\cdot\frac{L_{\nabla g} \hat \sigma_r}{\mu_g}.
    \]
\end{lemma}
\begin{proof}
Note that the output $\bar{r}^{(k+1)}$ of Algorithm \ref{algo: nhe} takes the following form
\begin{equation}\label{eq: hypergrad_expression_multi}
    \begin{aligned}
        r &= r_{0,x} - J_g\cdot \bar{r}_{M,y}, \\
        \bar{r}_{M,y}&= \alpha\cdot\sum_{i=0}^{M-1}\prod_{n=1}^{i}(I - \alpha H_{M+1-n})\cdot r_{0,y}.
    \end{aligned}
\end{equation}
Noting \eqref{ineq: epsilon_and_M}, definition of $\bar r_{*,y}$ in \eqref{def_er}, Neumann series, 
and under Assumption~\ref{assump: g}, we have
\begin{align}
    \|\bar{r}_{*,y}\| &\leq  \|\left[\nabla_{yy}^2g(x, y)\right]^{-1}\|\cdot \|\E\left[r_{0,y}\right]\|\leq \frac{\hat \sigma_r}{\mu_g},\label{ineq: bar_r_star_bound}\\
    \E\left[\bar{r}_{M,y}\right] &=\alpha\left[\sum_{n=0}^{M-1}(I-\alpha \nabla_{yy}^2 g(x, y))^n\right]\E\left[r_{0,y}\right] \notag\\
    & =\left[I - \left(I - \alpha \nabla_y^2 g(x, y) \right)^M\right]\left[\nabla_y^2g(x, y)\right]^{-1}\E\left[r_{0,y}\right] \notag\\
     &= \bar{r}_{*,y} - \left(I - \alpha \nabla_y^2 g(x, y) \right)^M\bar{r}_{*,y}.\label{eq: bar_r_M_y}
\end{align}

Moreover, by \eqref{eq: hypergrad_expression_multi} and \eqref{eq: bar_r_M_y}, we have
    \begin{equation}\label{eq: algo1_output_multi}
    \begin{aligned}
        \E\left[r\right] = \E\left[r_{0,x} - J_g\cdot \bar{r}_{M,y}\right]= \E\left[r_{0,x}\right] - \nabla_{xy}^2 g(x, y)\left[\nabla_y^2g(x, y)\right]^{-1}\E\left[r_{0,y}\right] + \mathcal{E},
    \end{aligned}
\end{equation}
where $\mathcal{E}$ defined in \eqref{def_er}. 
Under Assumption \ref{assump: g} and by \eqref{ineq: bar_r_star_bound}, we have
    \begin{equation}\label{ineq: hyper_error_multi}
        \|\mathcal{E}\|\leq L_{\nabla g}(1-\alpha\mu_g)^M\cdot \frac{\hat \sigma_r}{\mu_g},
    \end{equation}
    which together with \eqref{eq: algo1_output_multi}, complete the proof.
\end{proof}

Next, we prove the boundedness of some error terms that will be later used in our convergence analysis.
\begin{lemma}\label{lem: z_x_and_delta_d}
    Suppose Assumption \ref{assump: fi_lips}, \ref{assump: stochastic_oracles}, and \ref{assump: g} hold. Then in Algorithm \ref{algo: bilinasa} we have 
    \begin{align*}
        &\beta^2\E\left[\|z_k - x_k\|^2\right] \leq \E\left[\|d_k\|^2\right]\leq \sigma_w^2,\quad \E\left[\|d_{k+1}-d_k\|^2\right] \leq 4\tau_k^2\sigma_w^2,\ \text{for all } k\geq 0
    \end{align*}
\end{lemma}
\begin{proof}
    Note that $z_k$ in Algorithm \ref{algo: bilinasa} can be written as $
        z_k = \Pi_X(x_k - \frac{1}{\beta}d_k)$.
    The optimality condition of the projection gives $
        \<d_k + \beta(z_k - x_k), z - z_k>\geq 0, \forall z\in X$. Setting $z = x_k$ and using Cauchy-Schwartz inequality, we note that $
        \beta\|z_k - x_k\|\leq \|d_k\|$,
    which is due to the nonexpansiveness of projection operator. Then we know
    \begin{align*}
        \beta^2\E\left[\|z_k-x_k\|^2\right]\leq \E\left[\|d_k\|^2\right] 
        \leq\max\left(\E\left[\|d_{k-1}\|^2\right], \E\left[\|w_k\|^2\right]\right)\leq \max_{i\leq k}\E\left[\|w_i\|^2\right]\leq \sigma_w^2.
    \end{align*}
    The second inequality uses the fact that $d_k$ is a convex combination of $d_{k-1}$ and $w_k$, the third inequality applies the second inequality on each $\E\left[\|d_i\|^2\right]$, and the fourth ienquality uses Lemma \ref{lem: bdd_var}. Hence the first conclusion is proved. For $\|d_{k+1} - d_k\|$ we have:
    \begin{align*}
        \E\left[\|d_{k+1} - d_k\|^2\right] = \tau_k^2\E\left[\|d_k - w_{k+1}\|^2\right]\leq 2\tau_k^2(\E\left[\|d_k\|^2\right] + \E\left[\|w_{k+1}\|^2\right])\leq 4\tau_k^2\sigma_w^2,
    \end{align*}
    which completes the proof.
\end{proof}

For each $\|u_{k+1}^{(i)} - f_i(u_{k+1}^{(i+1)})\|$ we adopt Lemma 3.1 from \cite{balasubramanian2022stochastic}:
\begin{lemma}\label{lem: u_fu_decrease}
    Suppose Assumption \ref{assump: fi_lips} and \ref{assump: stochastic_oracles} hold. Define 
    \begin{align*}
        \theta_{k+1}^{(i)} &:=2\tau_k\langle \eta_{k+1}^{(i)}, E_{k,i} + (1-\tau_k)(f_i(u_k^{(i+1)}) - u_k^{(i)}) + (\hat \eta_{k+1}^{(i)})\T(u_{k+1}^{(i+1)} - u_k^{(i+1)})\rangle  \nonumber \\ 
        &+2\langle (\hat \eta_{k+1}^{(i)})\T(u_{k+1}^{(i+1)} - u_k^{(i+1)}), E_{k,i}+(1-\tau_k)(f_i(u_k^{(i+1)})-u_k^{(i)})\rangle,\\
        \hat{\theta}_{k+1}^{(i)} &:= \tau_k \langle -\eta_{k+1}^{(i)}, \tau_k(f_i(u_k^{(i+1)}) - u_k^{(i)}) + (J_{k+1}^{(i)})\T(u_{k+1}^{(i+1)}-u_k^{(i+1)})\rangle, \\
        \eta_{k+1}^{(i)} &:= f_i(u_k^{(i+1)}) - F_{k+1}^{(i)},\ \hat \eta_{k+1}^{(i)}: = \nabla f_i(u_k^{(i+1)}) - J_{k+1}^{(i)}, \\
        E_{k,i}  &:= f_i(u_{k+1}^{(i+1)}) - f_i(u_k^{(i+1)}) - \nabla f_i(u_k^{(i+1)})\T (u_{k+1}^{(i+1)}-u_k^{(i+1)}).
    \end{align*}
    Then in Algorithm \ref{algo: bilinasa} the following hold.
    \begin{itemize}
        \item [a)] For $1\leq i \leq T$, 
        \begin{equation}\label{ineq: u_and_fu_multi}
            \begin{aligned}
                \|u_{k+1}^{(i)} - f_i(u_{k+1}^{(i+1)})\|^2 \leq &(1 - \tau_k)\|u_k^{(i)} - f_i(u_k^{(i+1)})\|^2 + \tau_k^2 \|\eta_{k+1}^{(i)}\|^2 + \theta_{k+1}^{(i)} \\
                + &\left[4L_{\nabla f_i}^2 + \|f_i(u_{k}^{(i+1)}) - u_k^{(i)}\| +  \|\hat \eta_{k+1}^{(i)}\|^2\right]\|u_{k+1}^{(i+1)} - u_k^{(i+1)}\|^2,
            \end{aligned}
        \end{equation}
        
        \item [b)] For $1\leq i \leq T$,
        \begin{align*}
            \|u_{k+1}^{(i)} - u_k^{(i)}\|^2\leq \tau_k^2\left[2\| f_i(u_k^{(i+1)}) - u_k^{(i)}\|^2 + \|\eta_{k+1}^{(i)}\|^2 + \frac{2}{\tau_k^2}\|J_{k+1}^{(i)}\|^2 \| u_{k+1}^{(i+1)} - u_k^{(i+1)} \|^2\right]+2 \hat{\theta}_{k+1}^{(i)},
        \end{align*}
    \end{itemize}
\end{lemma}
\begin{proof}
    For $1\leq i\leq T$, by definition of $E_{k,i}, \hat \eta_{k+1}^{(i)},F_{k+1}^{(i)},u_{k+1}^{(i)}$, and $\theta_{k+1}^{(i)}$, we have
    \begin{align}
        & \|f_i(u_{k+1}^{(i+1)}) - u_{k+1}^{(i)}\|^2 \notag\\ 
        = & \|E_{k,i} + f_i(u_{k}^{(i+1)}) + \nabla f_i(u_{k}^{(i+1)})\T (u_{k+1}^{(i+1)}-u_{k}^{(i+1)}) - (1-\tau_k)u_k^{(i)} - \tau_kF_{k+1}^{(i)} - (J_{k+1}^{(i)})\T(u_{k+1}^{(i+1)}-u_{k}^{(i+1)})\|^2 \notag\\  
        = & \|E_{k,i} + (\hat \eta_{k+1}^{(i)})\T(u_{k+1}^{(i+1)}-u_{k}^{(i+1)}) + (1-\tau_k)(f_i(u_{k}^{(i+1)}) - u_k^{(i)}) + \tau_k\eta_{k+1}^{(i)}\|^2\notag \\  
        =& \|(\hat \eta_{k+1}^{(i)})\T(u_{k+1}^{(i+1)}-u_{k}^{(i+1)})\|^2 + \|E_{k,i} + (1-\tau_k)(f_i(u_{k}^{(i+1)})-u_k^{(i)})\|^2 +\tau_k^2\|\eta_{k+1}^{(i)}\|^2 + \theta_{k+1}^{(i)} \notag\\
        \leq & \|E_{k,i} + (1-\tau_k)(f_i(u_{k}^{(i+1)}) - u_k^{(i)})\|^2 +\tau_k^2\|\eta_{k+1}^{(i)}\|^2 + \theta_{k+1}^{(i)} + \|\hat \eta_{k+1}^{(i)}\|^2\|u_{k+1}^{(i+1)}-u_{k}^{(i+1)}\|^2\notag\\
        \leq & (1-\tau_k)\|f_i(u_{k}^{(i+1)}) - u_k^{(i)}\|^2+\|E_{k,i}\|^2 + 2(1-\tau_k)\<E_{k,i}, (f_i(u_{k}^{(i+1)}) - u_k^{(i)})>\notag \\
        &\qquad +\tau_k^2\|\eta_{k+1}^{(i)}\|^2 + \theta_{k+1}^{(i)}+\|\hat \eta_{k+1}^{(i)}\|^2\|u_{k+1}^{(i+1)}-u_{k}^{(i+1)}\|^2.\label{fi_wi_modified_nasa2}
    \end{align}
    where the second inequality holds by convexity of $\|\cdot\|^2$. By Assumption \ref{assump: fi_lips}, we have
    \begin{equation}\label{ineq: A_ki}
        \|E_{k,i}\| \le \frac{1}{2}\min\left\{4 L_{f_i}\|u_{k+1}^{(i+1)}-u_{k}^{(i+1)}\|, L_{\nabla f_i}\|u_{k+1}^{(i+1)}-u_{k}^{(i+1)}\|^2 \right\},
    \end{equation}
    and using Cauchy–Schwarz inequality in \eqref{fi_wi_modified_nasa2}, we obtain \eqref{ineq: u_and_fu_multi}.
    To show part b), noting definition of $\eta_{k+1}^{(i)}, \hat \eta_{k+1}^{(i)}$ and $\hat{\theta}_{k+1}^{(i)}$, Cauchy-Schwartz and Young's inequality, for $1\leq i \leq T$,
    \begin{align*}
        &\| u_{k+1}^{(i)} - u_k^{(i)}\|^2 \\
        =&\|\tau_k(F_{k+1}^{(i)} - f_i(u_k^{(i+1)})) + \tau_k(f_i(u_k^{(i+1)}) - u_k^{(i)}) + (J_{k+1}^{(i)})\T(u_{k+1}^{(i+1)} - u_{k}^{(i+1)}))\|^2 \\
        =&\tau_k^2\|\eta_{k+1}^{(i)}\|^2 + \|\tau_k(f_i(u_k^{(i+1)}) - u_k^{(i)}) + (J_{k+1}^{(i)})\T(u_{k+1}^{(i+1)} - u_{k}^{(i+1)}))\|^2 \\
        + &2\tau_k \< -\eta_{k+1}^{(i)}, \tau_k(f_i(u_{k}^{(i+1)}) - u_k^{(i)}) + (J_{k+1}^{(i)})\T(u_{k+1}^{(i+1)}-u_{k}^{(i+1)})> \\
        \leq &2\tau_k^2 \| f_i(u_{k}^{(i+1)}) - u_k^{(i)}\|^2 + \tau_k^2 \|\eta_{k+1}^{(i)}\|^2 + 2\|J_{k+1}^{(i)}\|^2 \|u_{k+1}^{(i+1)} - u_{k}^{(i+1)}\|^2 + 2 \hat{\theta}_{k+1}^{(i)}.
    \end{align*}
\end{proof}
Hence we know the decrease of $\|u_k^{(i)} - f_i(u_k^{(i+1)})\|^2$ for $1\leq i\leq T$:
\begin{align}
    &\|u_{k+1}^{(i)} - f_i(u_{k+1}^{(i+1)})\|^2 - \|u_k^{(i)} - f_i(u_k^{(i+1)})\|^2 \leq -\tau_k\|u_k^{(i)} - f_i(u_k^{(i+1)})\|^2 + \tilde{\theta}_{k+1}^{(i)}, \label{ineq: ufu_recursion}\\
    &\tilde{\theta}_{k+1}^{(i)} = \left[4L_{\nabla f_i}^2 + \|f_i(u_{k}^{(i+1)}) - u_k^{(i)}\| +  \|\hat \eta_{k+1}^{(i)}\|^2\right] \|u_{k+1}^{(i+1)} - u_{k}^{(i+1)}\|^2 +\tau_k^2\|\eta_{k+1}^{(i)}\|^2 + \theta_{k+1}^{(i)}. \label{eq: hat_r}
\end{align}
We adopt Lemma 3.2 in \cite{balasubramanian2022stochastic} to characterize $\| u_{k+1}^{(i)} - f_i(u_{k+1}^{(i+1)})\|^2$ and $\|u_{k+1}^{(i)} - u_k^{(i)}\|^2$.
\begin{lemma}\label{lem: uu_ufu_bound}
    Suppose Assumption \ref{assump: fi_lips}, \ref{assump: g}, \ref{assump: stochastic_oracles} and \ref{assump: stepsize} hold. In Algorithm \ref{algo: bilinasa} we have
    \begin{align}
        &\E\left[\|u_{k+1}^{(i)} - u_k^{(i)}\|^2|\F_k\right]\leq a_i\tau_k^2,\ \E\left[\|u_{k+1}^{(T+1)} - u_k^{(T+1)}\|^2|\F_k\right]\leq a_{T+1}\tau_k^2, \label{dif_u}\\ 
        &\E\left[\|u_{k}^{(i)} - f_i(u_{k}^{(i+1)})\|^2\right]\leq b_{i}^2 := \E\left[\|u_0^{(i)} - f_i(u_0^{(i+1)})\|^2\right] + \sigma_{F_i}^2+ (4L_{f_i}^2 + \hat \sigma_{J_i}^2)a_{i+1},\label{dif_uf}
    \end{align}
    for $1\leq i\leq T$. The constants are defined as
    \begin{equation}\label{ineq: A_B_const}
        \begin{aligned}
            a_i := &2b_i + \sigma_{F_i}^2 + 2\hat \sigma_{J_i}^2a_{i+1},\ b_i\geq 0,\\
            a_{T+1} := &\frac{\sigma_w^2}{\beta^2} + N^2c_{\gamma}^2\sigma_v^2 + N^2c_{\gamma}^2L_{\nabla g}\left[\|y_{1}^{(0)} - y_1^*\|^2 + \frac{2\sigma_v^2c_{\gamma}}{\mu_g} + \frac{2L_{y^*}^2\sigma_w^2}{\beta^2}\right] +\frac{N^3c_{\gamma}^4L_{\nabla g}\sigma_v^2}{2}, \\
        \end{aligned}
    \end{equation}
\end{lemma}
\begin{proof}
    Recall definitions of $E_{k,i}, \eta_{k+1}^{(i)}, \hat \eta_{k+1}^{(i)}$, and for $1\leq i\leq T$, define
    \[
        \Lambda_{k,i} = E_{k,i} + \tau_k\eta_{k+1}^{(i)} + \hat \eta_{k+1}^{(i)}(u_{k+1}^{(i+1)}-u_{k}^{(i+1)}).
    \]
    Then we know for $1\leq i\leq T$,
    \[
        u_{k+1}^{(i)} - f_i(u_{k+1}^{(i+1)}) = (1-\tau_k)(u_k^{(i)} - f_i(u_k^{(i+1)})) - \Lambda_{k,i}.
    \]
    Hence by convexity of $\|\cdot\|^2$ we know
    \begin{equation}\label{ineq: u_fu_decrease}
        \|u_{k+1}^{(i)} - f_i(u_{k+1}^{(i+1)})\|^2\leq (1-\tau_k)\|u_{k}^{(i)} - f_i(u_{k}^{(i+1)})\|^2 + \frac{1}{\tau_k}\|\Lambda_{k,i}\|^2.
    \end{equation}
    For $\Lambda_{k,i}$ we have
    \begin{align*}
        &\|\Lambda_{k,i}\|^2 = \|E_{k,i}\|^2 + \tau_k^2 \|\eta_{k+1}^{(i)}\|^2 + \|(\hat \eta_{k+1}^{(i)})\T(u_{k+1}^{(i+1)}-u_{k}^{(i+1)})\|^2 +2 \theta_{k,i}',\\
        &\theta_{k,i}' = \langle E_{k,i},\tau_k \eta_{k+1}^{(i)}+(\hat \eta_{k+1}^{(i)})\T(u_{k+1}^{(i+1)}-u_{k}^{(i+1)}) \rangle + \tau_k \langle \eta_{k+1}^{(i)}, (\hat \eta_{k+1}^{(i)})\T(u_{k+1}^{(i+1)}-u_{k}^{(i+1)})\rangle,
    \end{align*}
    which together with $\E\left[\theta_{k,i}'|\F_k\right] = 0$ imply
    \begin{equation}\label{ineq: D_ki}
        \begin{aligned}
            &\E[\|\Lambda_{k,i}\|^2| \F_k] \\
            = &\E[\|E_{k,i}\|^2| \F_k] + \tau_k^2 \E[\|\eta_{k+1}^{(i)}\|^2| \F_k] +
        \E[\|\hat \eta_{k+1}^{(i)}(u_{k+1}^{(i+1)}-u_{k}^{(i+1)})\|^2| \F_k] \\
        \leq &\tau_k^2 \E[\|\eta_{k+1}^{(i)}\|^2| \F_k] +
        \left(4L_{f_i}^2+\E[\|\hat \eta_{k+1}^{(i)}\|^2|\F_k]\right) \E[\|u_{k+1}^{(i+1)}-u_{k}^{(i+1)}\|^2| \F_k], \\
        \leq &\tau_k^2\sigma_{F_i}^2 + \left(4L_{f_i}^2 + \hat \sigma_{J_i}^2\right)\E[\|u_{k+1}^{(i+1)}-u_{k}^{(i+1)}\|^2| \F_k],
        \end{aligned}
    \end{equation}
    where the inequality follows from \eqref{ineq: A_ki}. Hence we know by $\tau_{k+1}\leq \tau_k$, \eqref{ineq: u_fu_decrease} and \eqref{ineq: D_ki} that
    \begin{equation}
        \begin{aligned}
            &\frac{1}{\Gamma_{k+1}}\E\left[\|u_{k+1}^{(i)} - f_i(u_{k+1}^{(i+1)})\|^2\right]\\
            \leq &\frac{1}{\Gamma_k}\E\left[\|u_k^{(i)} - f_i(u_k^{(i+1)})\|^2\right] +\frac{\sigma_{F_i}^2\tau_k}{\Gamma_k} + \frac{(4L_{f_i}^2 + \hat \sigma_{J_i}^2)}{\tau_k\Gamma_{k+1}}\E\left[\|u_{k+1}^{(i+1)}-u_{k}^{(i+1)}\|^2\right],
        \end{aligned}
    \end{equation}
    which gives
    \begin{equation}\label{ineq: ufu_bound}
        \begin{aligned}
            &\frac{1}{\Gamma_{k+1}}\E\left[\|u_{k+1}^{(i)} - f_i(u_{k+1}^{(i+1)})\|^2\right]\\
            \leq &\E\left[\|u_0^{(i)} - f_i(u_0^{(i+1)})\|^2\right] + \sigma_{F_i}^2\sum_{j=0}^{k}\frac{\tau_j}{\Gamma_j} + (4L_{f_i}^2 + \hat \sigma_{J_i}^2)\sum_{j=0}^{k}\frac{\E\left[\|u_{j+1}^{(i+1)}-u_{j}^{(i+1)}\|^2\right]}{\tau_j\Gamma_{j+1}} \\
            \leq &\E\left[\|u_0^{(i)} - f_i(u_0^{(i+1)})\|^2\right] + \frac{\sigma_{F_i}^2}{\Gamma_{k+1}}+ (4L_{f_i}^2 + \hat \sigma_{J_i}^2)\sum_{j=0}^{k}\frac{\E\left[\|u_{j+1}^{(i+1)}-u_{j}^{(i+1)}\|^2\right]}{\tau_j\Gamma_{j+1}}
        \end{aligned}
    \end{equation}
    By Lemma \ref{lem: u_fu_decrease} we know
    \begin{equation}\label{ineq: uu_bound}
        \begin{aligned}
            &\E\left[\|u_{k+1}^{(i)} - u_k^{(i)}\|^2\right]\\
            \leq &\tau_k^2\E\left[2\|u_k^{(i)} - f_i(u_k^{(i+1)})\|^2 + \|\eta_{k+1}^{(i)}\|^2 + \frac{2}{\tau_k^2}\|J_{k+1}^{(i)}\|^2 \| u_{k+1}^{(i+1)} - u_k^{(i+1)} \|^2\right], \\
            \leq &2\tau_k^2\E\left[\|u_k^{(i)} - f_i(u_k^{(i+1)})\|^2\right] + \tau_k^2\sigma_{F_i}^2 + 2\hat \sigma_{J_i}^2\E\left[\| u_{k+1}^{(i+1)} - u_k^{(i+1)} \|^2\right].
        \end{aligned}
    \end{equation}
    Notice that by definition of $u_k^{(T+1)}$, we have $\|u_{k+1}^{(T+1)} - u_k^{(T+1)}\|^2 = \|x_{k+1}-x_k\|^2 + \|y_{k+1}^{(N)} - y_k^{(N)}\|^2$. Hence by using Lemma \ref{lem: inner_analysis} and \ref{lem: z_x_and_delta_d}, we know
    \begin{align*}
        &\E\left[\|u_{k+1}^{(T+1)} - u_k^{(T+1)}\|^2|\F_k\right]\\
        = &\E\left[\|x_{k+1}-x_k\|^2|\F_k\right] + \E\left[\|y_{k+1}^{(N)} - y_k^{(N)}\|^2|\F_k\right] \\
        \leq &\tau_k^2\frac{\sigma_w^2}{\beta^2} + N^2\gamma_{k+1}^2\sigma_v^2 + \min\left\{N\gamma_{k+1},\ \frac{1}{\mu_g}\right\}NL_{\nabla g}\gamma_{k+1} \E\left[\|y_{k+1}^{(0)} - y_{k+1}^*\|^2|\F_k\right] + \frac{N^3L_{\nabla g}\gamma_{k+1}^4\sigma_v^2}{2} \\
        \leq &\tau_k^2\frac{\sigma_w^2}{\beta^2} + N^2c_{\gamma}^2\sigma_v^2\tau_k^2 + N^2c_{\gamma}^2L_{\nabla g}\tau_k^2\left[\E\left[\|y_{1}^{(0)} - y_1^*\|^2\right] + \frac{2\sigma_v^2c_{\gamma}}{\mu_g} + \frac{2L_{y^*}^2\sigma_w^2}{\beta^2}\right] + \frac{N^3c_{\gamma}^4L_{\nabla g}\sigma_v^2\tau_k^4}{2} \\
        \leq &a_{T+1}\tau_k^2,
    \end{align*}
    where we use \eqref{ineq: inner_part1_1} in the first inequality and \eqref{ineq: inner_part2_1} in the second inequality. Combining the above inequality, \eqref{ineq: ufu_bound}, \eqref{ineq: uu_bound} and Lemma \ref{lem: ineq_of_seqs}, the proof is complete by using backward induction.
\end{proof}
As a direct result of the above lemma, we can now characterize $\sum_{k=0}^{K}\E\left[\tilde{\theta}_{k+1}^{(i)}\right]$. 
\begin{lemma}\label{lem: residual_and_Ek}
    Suppose that Assumptions \ref{assump: fi_lips}, \ref{assump: g}, \ref{assump: stochastic_oracles} and \ref{assump: stepsize} hold. Then, we have
    \begin{align*}                  \sum_{k=0}^{K}\E\left[\tilde{\theta}_{k+1}^{(i)}\right]\leq \left[(4L_{\nabla f_i}^2 + \sigma_{J_i}^2 + b_i)a_{i+1} + \sigma_{F_i}^2\right]\sum_{k=0}^{K}\tau_k^2.
    \end{align*}
\end{lemma}
\begin{proof}
    Recalling the definition of $\tilde{\theta}_{k+1}^{(i)}$ in \eqref{eq: hat_r}, we have
    \begin{align*}
        &\sum_{k=0}^{K}\E\left[\tilde{\theta}_{k+1}^{(i)}\right] \\
        = &\sum_{k=0}^{K}\E\left[\left(4L_{\nabla f_i}^2 + \|f_i(u_{k}^{(i+1)}) - u_k^{(i)}\| +  \|\hat \eta_{k+1}^{(i)}\|^2\right) \|u_{k+1}^{(i+1)} - u_{k}^{(i+1)}\|^2 +\tau_k^2\|\eta_{k+1}^{(i)}\|^2\right] \\
        \leq &\sum_{k=0}^{K}\left[(4L_{\nabla f_i}^2 + \sigma_{J_i}^2)a_{i+1} + \sigma_{F_i}^2\right]\tau_k^2 + \sum_{k=0}^{K}\E\left[\E\left[\|u_{k+1}^{(i+1)} - u_{k}^{(i+1)}\|^2|\F_k\right]\|f_i(u_{k}^{(i+1)}) - u_k^{(i)}\|\right], 
    \end{align*}
    which together with \eqref{dif_uf}, complete the proof.
\end{proof}
Now, we define the following merit function 
\begin{equation}\label{eq: merit_multi}
    W_k = \Phi(x_k) - \Phi^* - \eta(x_k,d_k,\beta) + \sum_{i=1}^{T}\rho_i\|u_k^{(i)} - f_i(u_k^{(i+1)})\|^2 + \nu\|y_{k}^{(0)} - y_{k}^*\|^2, 
\end{equation}
which will be used to combine the previous results.  We emphasize here that the above merit function is different from prior analyses of nested compositional problems, e.g.,~\cite{balasubramanian2022stochastic,ruszczynski2020convergence} in that it  is also designed to handle the additional bi-level structure. Recall the measure of optimality,$V_k$, as defined in~\eqref{eq: V_k}.  
The following result analyzes $V_k$ in terms of the above merit function.
\begin{lemma}\label{lem: decrease_of_any}
    Suppose that Assumptions \ref{assump: fi_lips}, \ref{assump: g}, \ref{assump: stochastic_oracles} and \ref{assump: stepsize} hold. Then, we have
    \begin{align*}
        \sum_{i=0}^{K}\tau_i\E\left[\|d_i - \nabla\Phi(x_{i})\|^2\right]=O_T\left(\sum_{i=0}^{K}\tau_i\E\left[\|z_i-x_i\|^2 + \sum_{j=2}^{T}\|f_j(u_i^{(j+1)}) - u_i^{(j)}\|^2 + \|y_i^{(0)} - y_i^*\|^2+ \delta_g^2\right] + 1\right).
    \end{align*}
\end{lemma}

\begin{proof}
    We first analyze the decrease of $\|d_k-\nabla\Phi(x_k)\|^2$. Noting Step 10 of Algorithm \ref{algo: bilinasa} and convexity of $\|\cdot\|^2$, we have
\begin{align}
    \|d_k - \nabla\Phi(x_{k})\|^2 &= \|(1 - \tau_{k-1})(d_{k-1} - \nabla\Phi(x_{k-1})) + \tau_{k-1}(e_{k-1} + \Delta_k^{F})\|^2 \notag\\
    &\leq (1-\tau_{k-1})\|(d_{k-1} - \nabla\Phi(x_{k-1}))\|^2 + \tau_{k-1}\|e_{k-1}\|^2 + \tau_{k-1}^2\|\Delta_k^F\|^2 \notag\\
    &+2\tau_{k-1}\<(1 - \tau_{k-1})(d_{k-1} - \nabla\Phi(x_{k-1})) + \tau_{k-1}e_{k-1}, \Delta_k^F>,\label{d-grad}
\end{align}
where
\begin{align}\label{eq: e_and_delta}
    \begin{aligned}
        e_{k-1} &= \frac{1}{\tau_{k-1}}(\nabla\Phi(x_{k-1}) - \nabla\Phi(x_{k})) + (\E\left[w_k|\F_{k-1}\right] - \nabla\Phi(x_{k-1})), \\
        \Delta_k^{F} &= w_k - \E\left[w_k|\F_{k-1}\right].
    \end{aligned}
\end{align}
Taking expectation on both sides of \eqref{d-grad}, we obtain
\begin{equation}\label{ineq: d_decrease_multi}
    \E\left[\|d_k - \nabla\Phi(x_{k})\|^2\right]\leq (1-\tau_{k-1})\E\left[\|(d_{k-1} - \nabla\Phi(x_{k-1}))\|^2\right] + \tau_{k-1} \E\left[\|e_{k-1}\|^2\right] + \tau_{k-1}^2\E\left[\|\Delta_k^F\|^2\right].
\end{equation}
Noting Lemma~\ref{lem: ineq_of_seqs}, defining
\begin{align*}
    &a_{k+1} = \E\left[\|d_k - \nabla\Phi(x_{k})\|^2\right],\quad b_k = \tau_{k-1} \E\left[\|e_{k-1}\|^2\right] + \tau_{k-1}^2\E\left[\|\Delta_k^F\|^2\right],\quad c_k = 1,\quad \delta_k = 1 - \tau_{k-1},
\end{align*}
and in the view of Lemma~\ref{lem: decrease_lemma}, for $k\geq 1$, we have 
\begin{align}\label{ineq: main_term1_multi}
    \begin{aligned}
        &\sum_{i=0}^{K}\tau_i\E\left[\|d_i - \nabla\Phi(x_{i})\|^2\right]\\
        \leq &\E\left[\|d_0 - \nabla\Phi(x_0)\|^2\right] + \sum_{i=1}^{K}\left(\tau_{i-1} \E\left[\|e_{i-1}\|^2\right] + \tau_{i-1}^2\E\left[\|\Delta_i^F\|^2\right]\right)\\
    \leq &\E\left[\|d_0 - \nabla\Phi(x_0)\|^2\right] + \sum_{i=0}^{K-1}\tau_i\E\left[\|e_i\|^2\right] + \sigma_w^2\sum_{i=0}^{K-1}\tau_i^2,
    \end{aligned}
\end{align}
where the second inequality follows from Lemma \ref{lem: bdd_var} and the definition of $\Delta_i^{F}$ in \eqref{eq: e_and_delta}. Now, to bound $\|e_i\|$, we need to first analyze $\E\left[r_0^{(k+1)}|\F_k\right] - \nabla\Psi(x_k,y_k^*)$. We adopt Lemma 2.4 from \cite{balasubramanian2022stochastic} that 
\begin{align*}
    \|\nabla \Psi(x,y) - \nabla f_T(x,y)\prod_{i=2}^{T}\nabla f_{T+1-i}(u^{(T+2-i)})\|\leq \sum_{j=2}^{T-1}C_j\|f_j(u^{(j+1)}) - u^{(j)}\| + C_T\|f_T(x,y) - u^{(T)}\|,
\end{align*}
where, according to \cite{balasubramanian2022stochastic}, the constants are defined as
\begin{align*}
    &R_1 = L_{\nabla f_1}L_{f_2}\cdots L_{f_T},\ R_j = L_{f_1}\cdots L_{f_{j-1}}L_{\nabla f_j}L_{f_{j+1}}\cdots L_{f_T}/L_{f_j}\quad 2\leq j\leq T-1, \\
    &C_2 = R_1,\ C_j = \sum_{i=1}^{j-2}R_i\left(\prod_{l=i+1}^{j-1}L_{f_l}\right)\quad 3\leq j\leq T.
\end{align*}
By replacing $r_0$ and $\mathcal{E}$ with $r_0^{(k+1)}$ and $\mathcal{E}_k$ in Algorithm \ref{algo: nhe} and Lemma \ref{lem: hypergrad_exp}, to represent their corresponding vectors when the inputs are $x_k, y_k^{(N)}$ and $u_k^{(i)} (1\leq i\leq T)$, we have
\begin{align*}
    \E\left[r_0^{(k+1)}|\F_k\right] = \nabla f_T(x_k,y_k^{(N)})\prod_{i=2}^{T}\nabla f_{T+1-i}(u_k^{(T+2-i)}),
\end{align*}
which together with the Lipschitz smoothness assumption on $\nabla \Psi$, imply that
\begin{align}\label{ineq: exp_and_psi}
    \begin{aligned}
        &\|\E\left[r_0^{(k+1)}|\F_k\right] - \nabla\Psi(x_k,y_k^*)\|\\
    \leq &\|\nabla f_T(x_k,y_k^{(N)})\prod_{i=2}^{T}\nabla f_{T+1-i}(u_k^{(T+2-i)}) - \nabla \Psi(x_k,y_k^{(N)})\| + \|\nabla \Psi(x_k,y_k^{(N)}) - \nabla\Psi(x_k,y_k^*)\|\\
    \leq &\sum_{j=2}^{T}C_j\|f_j(u_k^{(j+1)}) - u_k^{(j)}\| + L_{\nabla\Psi}\|y_k^{(N)} - y_k^*\|
    \end{aligned}
\end{align}
Define the positive constant $\tilde{C}$ as
\begin{align*}
    \tilde{C}^2 = \max\left\{L_{\Phi}^2, \max_{2\leq j\leq T}\left[\left(1 + \frac{L_{\nabla g}^2}{\mu_g^2}\right)C_j^2\right], \left(\sqrt{1 + \frac{L_{\nabla g}^2}{\mu_g^2}}L_{\nabla\Psi} + \frac{(L_{\nabla g}+\mu_g)L_{\nabla^2 g}}{\mu_g^2}\right)^2\right\}
\end{align*}
From \eqref{ineq: exp_and_psi} and Lemma \ref{lem: hypergrad_exp}, we obtain
\begin{align}\label{eq: w_k_decom_multi}
    \begin{aligned}
        &\|\E\left[w_{k+1}|\F_k\right] - \nabla\Phi(x_{k})- \mathcal{E}_k\|\\
        = &\|\E\left[r_{0,x}^{(k+1)}|\F_k\right] -\nabla_{xy}^2 g(x_k, y_k^{(N)})\left[\nabla_y^2g(x_k,y_k^{(N)})\right]^{-1}\E\left[r_{0,y}^{(k+1)}|\F_k\right]\\
        & - \nabla_x \Psi(x_k, y_k^*)+ \nabla_{xy}^2 g(x_k, y_k^*)\left[\nabla_y^2g(x_k,y_k^*)\right]^{-1}\nabla_y \Psi(x_k, y_k^*)\|\\
        \leq&\|\begin{pmatrix}
            I, -\nabla_{xy}^2 g(x_k, y_k^{(N)})\left[\nabla_y^2g(x_k,y_k^{(N)})\right]^{-1}
        \end{pmatrix}\left(\E\left[r_0^{(k+1)}|\F_k\right] - \nabla\Psi(x_k,y_k^*)\right)\| \\
        + &\|\nabla_{xy}^2 g(x_k, y_k^*)\left[\nabla_y^2g(x_k,y_k^*)\right]^{-1}\nabla\Psi_y(x_k,y_k^*) -\nabla_{xy}^2 g(x_k, y_k^{(N)})\left[\nabla_y^2g(x_k,y_k^{(N)})\right]^{-1}\nabla\Psi_y(x_k,y_k^*)\| \\
        \leq & \sqrt{1 + \frac{L_{\nabla g}^2}{\mu_g^2}}\left[\sum_{j=2}^{T}C_j\|f_j(u_k^{(j+1)}) - u_k^{(j)}\| + L_{\nabla\Psi}\|y_k^{(N)} - y_k^*\|\right] +\frac{(L_{\nabla g} + \mu_g)L_{\nabla^2 g}}{\mu_g^2}\|y_k^{(N)} - y_k^*\| \\
        \leq &\tilde{C}\left(\sum_{j=2}^{T}\|f_j(u_k^{(j+1)}) - u_k^{(j)}\| + \|y_k^{(N)} - y_k^*\|\right),
    \end{aligned}
\end{align}
where the third inequality follows from \eqref{ineq: exp_and_psi} and the fact that
\begin{align*}
    &\|\nabla_{xy}^2 g(x_k, y_k^*)\left[\nabla_y^2g(x_k,y_k^*)\right]^{-1}-\nabla_{xy}^2 g(x_k, y_k^{(N)})\left[\nabla_y^2g(x_k,y_k^{(N)})\right]^{-1}\| \\
    \leq &\|\nabla_{xy}^2 g(x_k, y_k^*)\left[\nabla_y^2g(x_k,y_k^*)\right]^{-1}-\nabla_{xy}^2 g(x_k, y_k^{(N)})\left[\nabla_y^2g(x_k,y_k^*)\right]^{-1}\| \\
    + &\|\nabla_{xy}^2 g(x_k, y_k^{(N)})\left[\nabla_y^2g(x_k,y_k^*)\right]^{-1}-\nabla_{xy}^2 g(x_k, y_k^{(N)})\left[\nabla_y^2g(x_k,y_k^{(N)})\right]^{-1}\| \\
    \leq &\frac{L_{\nabla^2g}}{\mu_g}\|y_k^{(N)} - y_k^*\|+L_{\nabla_g}\|\left[\nabla_y^2g(x_k,y_k^*)\right]^{-1}(\nabla_y^2g(x_k,y_k^{(N)}) -  \nabla_y^2g(x_k,y_k^*))\left[\nabla_y^2g(x_k,y_k^{(N)})\right]^{-1}\| \\
    \leq &\frac{(L_{\nabla g} + \mu_g)L_{\nabla^2 g}}{\mu_g^2}\|y_k^{(N)} - y_k^*\|.
\end{align*}
Inequality \eqref{eq: w_k_decom_multi} indicates that
\begin{equation}\label{ineq: E_w_and_nabla_phi}
    \|\E\left[w_{k+1}|\F_k\right] - \nabla\Phi(x_{k})\|\leq \tilde{C}\left(\sum_{j=2}^{T}\|f_j(u_k^{(j+1)}) - u_k^{(j)}\| + \|y_k^{(N)} - y_k^*\|\right) + \|\mathcal{E}_k\|,
\end{equation}
which, together with the definition of $e_i$ in \eqref{eq: e_and_delta}, imply that $\|e_i\|^2\leq \tilde{C}^2(T+2)(\|z_i-x_i\|^2 + \sum_{j=2}^{T}\|f_j(u_i^{(j+1)}) - u_i^{(j)}\|^2 + \|y_i^{(N)} - y_i^*\|^2) + (T+2)\|\mathcal{E}_i\|^2$.
    
By the above inequality and \eqref{ineq: main_term1_multi}, we have
\begin{align}
    &\sum_{i=0}^{K}\tau_i\E\left[\|d_i - \nabla\Phi(x_{i})\|^2\right]\label{ineq: newbound_d_phi}\\
    \leq &\E\left[\|d_0 - \nabla\Phi(x_0)\|^2\right] + \sum_{i=0}^{K-1}\tau_i\E\left[\|e_i\|^2\right] + \sigma_w^2\sum_{i=0}^{K-1}\tau_i^2 \notag\\
    \leq &\tilde{C}^2(T+2)\sum_{i=0}^{K}\tau_i\E\left[\|z_i-x_i\|^2 + \sum_{j=2}^{T}\|f_j(u_i^{(j+1)}) - u_i^{(j)}\|^2 + \|y_i^{(N)} - y_i^*\|^2\right] \notag\\
    + &\frac{L_{\nabla g}^2 \hat \sigma_r^2}{\mu_g^2}(T+2)\delta_g^2\sum_{i=0}^{K}\tau_i + \E\left[\|d_0 - \nabla\Phi(x_0)\|^2\right] + \sigma_w^2\sum_{i=0}^{K-1}\tau_i^2 \notag\\
    = & O_T\left(\sum_{i=0}^{K}\tau_i\E\left[\|z_i-x_i\|^2 + \sum_{j=2}^{T}\|f_j(u_i^{(j+1)}) - u_i^{(j)}\|^2 + \|y_i^{(0)} - y_i^*\|^2 + \delta_g^2\right] + 1\right).\notag
\end{align}
The equality holds due to Lemma \ref{lem: inner_analysis} and under Assumption \ref{assump: stepsize}.
\end{proof}

We are now ready to provide the proof of Theorem~\ref{thm:mainresult}.
\begin{proof}[{\bf Proof of Theorem~\ref{thm:mainresult}}]
Observe that
\begin{align}
    &\E\left[\<\nabla\Phi(x_k) - w_{k+1}, z_k - x_k>\right] = \E\left[\E\left[\<\nabla\Phi(x_k) - w_{k+1}, z_k - x_k>|\F_k\right]\right] \label{ineq: hypergrad_dot_error} \\
        =&\E\left[\<\nabla\Phi(x_k)-\E\left[w_{k+1}|\F_k\right], z_k-x_k>\right]\leq \E\left[\|\nabla\Phi(x_k)-\E\left[w_{k+1}|\F_k\right]\|\cdot \|z_k-x_k\|\right] \notag\\
        \leq & \E\left[\tilde{C}\sum_{j=2}^{T}\|f_j(u_k^{(j+1)}) - u_k^{(j)}\|\|z_k-x_k\| + \tilde{C}\|y_k^{(N)}-y_k^*\|\|z_k-x_k\| + \|\mathcal{E}_k\| \|z_k-x_k\|\right] \notag\\
        \leq &\tilde{C}\E\left[\sum_{j=2}^{T}\|f_j(u_k^{(j+1)}) - u_k^{(j)}\|\|z_k-x_k\| + (\|y_k^{(0)}-y_k^*\| + \sqrt{N}\gamma_k\sigma_v)\|z_k-x_k\|\right] \notag\\
        + & \E\left[\frac{1}{4\lambda}\|\mathcal{E}_k\|^2 + \lambda\|z_k-x_k\|^2\right],\notag
\end{align}
where $\lambda$ is a positive constant to be determined, the second inequality follows from \eqref{ineq: E_w_and_nabla_phi}, the third one follows from Lemma \ref{lem: inner_analysis}, and the fact that
\begin{align*}
    &\E\left[\|y_k^{(N)}-y_k^*\|\|z_k-x_k\|\right] 
    \\
    = &\E\left[\E\left[\|y_k^{(N)}-y_k^*\|\|z_k-x_k\||x_k,z_k,y_k^{(0)}\right]\right] = \E\left[\E\left[\|y_k^{(N)} - y_k^*\||x_k,y_k^{(0)}\right]\|z_k-x_k\|\right] \\
    \leq &\E\left[\sqrt{\E\left[\|y_k^{(N)}-y_k^*\|^2|x_k,y_k^{(0)}\right]}\|z_k-x_k\|\right]\leq \E\left[\sqrt{\|y_k^{(0)}-y_k^*\|^2 + N\gamma_k^2\sigma_v^2}\|z_k-x_k\|\right]\\
    \leq&\E\left[(\|y_k^{(0)}-y_k^*\|+\sqrt{N}\gamma_k\sigma_v)\|z_k-x_k\|\right].
\end{align*}
We also have
\begin{align}
    &\Phi(x_{k+1}) - \Phi(x_k)\leq \frac{L_{\Phi}\tau_k^2}{2}\|z_k - x_k\|^2 + \tau_k\<\nabla\Phi(x_k), z_k - x_k>, \label{ineq: delta_phi_multi}\\
    &\eta(x_k,d_k,\beta) - \eta(x_{k+1},d_{k+1},\beta) \label{ineq: delta_eta_multi}\\
    \leq &-\beta\tau_k\|z_k - x_k\|^2 - \tau_k\<w_{k+1}, z_k-x_k> + \frac{L_{\nabla\eta}}{2}\left[\|x_{k+1}-x_k\|^2 + \|d_{k+1} - d_k\|^2\right], \notag\\
    &\E\left[\|y_{k+1}^{(0)} - y_{k+1}^*\|^2\right] - \E\left[\|y_{k}^{(0)} - y_{k}^*\|^2\right]\leq-\tau_k\E\left[\|y_{k}^{(0)} - y_{k}^*\|^2\right] + 2N\gamma_{k}^2\sigma_v^2 + 2L_{y^*}^2\tau_{k}\E\left[\|x_{k}-z_{k}\|^2\right],\label{ineq: delta_inner_multi}
\end{align}
where \eqref{ineq: delta_phi_multi} uses $L_{\Phi}$-smoothness of $\Phi$, \eqref{ineq: delta_eta_multi} uses $L_{\nabla\eta}$-smoothness of $\eta$ and the definition of $z_k$ (see (3.9) of \cite{ghadimi2020single}), and \eqref{ineq: delta_inner_multi} is from \eqref{ineq: delta_inner_multi_original}. Now, we are ready to bound $W_{k+1} - W_k$. By \eqref{eq: merit_multi}, \eqref{ineq: delta_phi_multi}, \eqref{ineq: delta_eta_multi}, \eqref{ineq: delta_inner_multi}, and Lemma \ref{lem: u_fu_decrease} we have
\begin{align}
    &\E\left[W_{k+1} - W_k\right] \label{ineq: W_decrease}\\
    \leq & \E\left[\frac{L_{\Phi}\tau_k^2}{2}\|z_k - x_k\|^2 + \tau_k\<\nabla\Phi(x_{k}), z_k - x_k>\right] \notag\\
    - &\E\left[\beta\tau_k\|z_k - x_k\|^2 + \tau_k\<w_{k+1}, z_k-x_k>\right] + \frac{L_{\nabla\eta}}{2}\E\left[\|x_{k+1}-x_k\|^2 + \|d_{k+1} - d_k\|^2\right] \notag\\
    - &\sum_{i=1}^{T}\rho_i\tau_k\E\left[\|u_k^{(i)} - f_i(u_k^{(i+1)})\|^2\right] +\sum_{i=1}^{T}\rho_i\E\left[\tilde{\theta}_{k+1}^{(i)}\right] + \nu\left(\E\left[\|y_{k+1}^* - y_{k+1}^{(0)}\|^2\right] - \E\left[\|y_{k}^* - y_{k}^{(0)}\|^2\right]\right) \notag\\
    = & - \beta\tau_k\E\left[\|z_k - x_k\|^2\right] +  \nu\left(\E\left[\|y_{k+1}^* - y_{k+1}^{(0)}\|^2\right] - \E\left[\|y_{k}^* - y_{k}^{(0)}\|^2\right]\right) \notag\\
    - &\sum_{i=1}^{T}\rho_i\tau_k\E\left[\|u_k^{(i)} - f_i(u_k^{(i+1)})\|^2\right] + \tau_k\E\left[\<\nabla\Phi(x_k) - w_{k+1}, z_k - x_k>\right] \notag\\
    + &\frac{(L_{\nabla\eta}+L_{\Phi})\tau_k^2}{2}\E\left[\|z_k-x_k\|^2\right] + \frac{L_{\nabla\eta}}{2}\E\left[\|d_{k+1}-d_k\|^2\right] + \sum_{i=1}^{T}\rho_i\E\left[\tilde{\theta}_{k+1}^{(i)}\right] \notag\\
    \leq &- \beta\tau_k\E\left[\|z_k - x_k\|^2\right] -\nu\tau_k\E\left[\|y_{k}^{(0)} - y_{k}^*\|^2\right] + 2\nu N\gamma_{k}^2\sigma_v^2 + 2\nu L_{y^*}^2\tau_{k}\E\left[\|x_{k}-z_{k}\|^2\right] \notag\\
    -& \sum_{i=1}^{T}\rho_i\tau_k\E\left[\|u_k^{(i)} - f_i(u_k^{(i+1)})\|^2\right] + \tau_k\E\left[\tilde{C}\sum_{j=2}^{T}\|f_j(u_k^{(j+1)}) - u_k^{(j)}\|\|z_k-x_k\|\right] \notag\\
    +&\tau_k\E\left[\tilde{C}(\|y_k^{(0)}-y_k^*\| + \sqrt{N}\gamma_k\sigma_v)\|z_k-x_k\| +\frac{1}{4\lambda}\|\mathcal{E}_k\|^2 + \lambda\|z_k-x_k\|^2\right] \notag\\
    +&\frac{(L_{\nabla\eta}+L_{\Phi})\tau_k^2}{2}\E\left[\|z_k-x_k\|^2\right] + \frac{L_{\nabla\eta}}{2}\E\left[\|d_{k+1}-d_k\|^2\right] + \sum_{i=1}^{T}\rho_i\E\left[\tilde{\theta}_{k+1}^{(i)}\right]. \notag
\end{align}
Assume that we choose the constants $\beta, \nu, \rho_1,...,\rho_T, \lambda$ such that 
\begin{align}\label{ineq: quadratic_form}
    \begin{aligned}
        -&(\beta - \lambda - 2\nu L_{y^*}^2)\|z_k-x_k\|^2 - \nu\|y_k^{(0)}-y_k^*\|^2 - \sum_{i=1}^{T}\rho_i\|u_k^{(i)} - f_i(u_k^{(i+1)})\|^2 \\
    + &\tilde{C}\|z_k-x_k\|\|y_k^{(0)} - y_k^*\| + \tilde{C}\sum_{j=2}^{T}\|z_k-x_k\|\|f_j(u_k^{(j+1)}) - u_k^{(j)}\| \\
    \leq &-c\cdot\left(\|z_k-x_k\|^2 + \|y_k^{(0)}-y_k^*\|^2 + \sum_{i=1}^{T}\|u_k^{(i)} - f_i(u_k^{(i+1)})\|^2\right)
    \end{aligned}
\end{align}
for some constant $c>0$, define
\begin{align}\label{eq: R_k}
    \begin{aligned}
        R_k &= 2\nu N\gamma_{k}^2\sigma_v^2 + \tilde{C}\sqrt{N}\sigma_v\tau_k\gamma_k\E\left[\|z_k-x_k\|\right] + \frac{\tau_k}{4\lambda}\E\left[\|\mathcal{E}_k\|^2\right] \\
        &+ \frac{(L_{\nabla\eta}+L_{\Phi})\tau_k^2}{2}\E\left[\|z_k-x_k\|^2\right] + \frac{L_{\nabla\eta}}{2}\E\left[\|d_{k+1}-d_k\|^2\right] + \sum_{i=1}^{T}\rho_i\E\left[\tilde{\theta}_{k+1}^{(i)}\right].
    \end{aligned}
\end{align}
and the constant
\begin{align}
    \begin{aligned}
        C_{R,1} &= 2\nu Nc_{\gamma}^2\sigma_v^2 + \tilde{C}\sqrt{N}c_{\gamma}\sigma_v\frac{\sigma_w}{\beta} + \frac{(L_{\nabla\eta}+L_{\Phi})\sigma_w^2}{2\beta^2} + 2L_{\nabla \eta}\sigma_w^2 \\
        &+ \sum_{i=1}^{T}\rho_i\left[(4L_{\nabla f_i}^2 + \sigma_{J_i}^2 + b_i)a_{i+1} + \sigma_{F_i}^2\right].
    \end{aligned}
\end{align}
Taking summation on both sides of \eqref{eq: R_k}, we obtain
\begin{align}\label{ineq: sum_Rk}
    \begin{aligned}
        \sum_{k=0}^{K}R_k\leq &2\nu Nc_{\gamma}^2\sigma_v^2\sum_{k=0}^{K}\tau_k^2 + \tilde{C}\sqrt{N}c_{\gamma}\sigma_v\frac{\sigma_w}{\beta}\sum_{k=0}^{K}\tau_k^2 + \frac{1}{4\lambda}\sum_{k=0}^{K}\tau_k\E\left[\|\mathcal{E}_k\|^2\right] \\
        + &\frac{(L_{\nabla\eta}+L_{\Phi})\sigma_w^2}{2\beta^2}\sum_{k=0}^{K}\tau_k^2 + 2L_{\nabla \eta}\sigma_w^2\sum_{k=0}^{K}\tau_k^2 \\
        + &\sum_{i=1}^{T}\rho_i\left[(4L_{\nabla f_i}^2 + \sigma_{J_i}^2 + b_i)a_{i+1} + \sigma_{F_i}^2\right]\sum_{k=0}^{K}\tau_k^2 \leq C_{R,1}\sum_{k=0}^{K}\tau_k^2 + \frac{\delta_g^2}{4\lambda}\sum_{k=0}^{K}\tau_k,
    \end{aligned}
\end{align}
where the second inequality follows from Lemma \ref{lem: hypergrad_exp}. Taking summation on both sides of \eqref{ineq: W_decrease} and using \eqref{ineq: quadratic_form}, \eqref{eq: R_k} and \eqref{ineq: sum_Rk}, we have
\begin{align*}
    &\E\left[W_{K+1}\right] - \E\left[W_0\right]\\
    \leq &-c\cdot\sum_{k=0}^{K}\tau_k\E\left[\|z_k-x_k\|^2 + \|y_k^{(0)}-y_k^*\|^2 + \sum_{i=1}^{T}\|u_k^{(i)} - f_i(u_k^{(i+1)})\|^2\right] \\
    + &C_{R,1}\sum_{k=0}^{K}\tau_k^2 + \frac{\delta_g^2}{4\lambda}\sum_{k=0}^{K}\tau_k,
\end{align*}
which implies that
\begin{align}\label{ineq: merit_sum}
    \begin{aligned}
        &\sum_{k=0}^{K}\tau_k\E\left[\|z_k-x_k\|^2 + \|y_k^{(0)}-y_k^*\|^2 + \sum_{i=1}^{T}\|u_k^{(i)} - f_i(u_k^{(i+1)})\|^2\right]\\
        \leq &\frac{1}{c}\left[\E\left[W_0 - W_{K+1}\right] + C_{R,1}\sum_{k=0}^{K}\tau_k^2 + \frac{\delta_g^2}{4\lambda}\sum_{k=0}^{K}\tau_k\right].
    \end{aligned}
\end{align}
Now, if we choose $R\in\{0, 1,2,...,K\}$ satisfying $P(R=k) = \frac{\tau_k}{\sum_{j=0}^{K}\tau_j}$, we have:
\begin{align*}
    \E\left[V(x_R,d_R)\right]\leq \frac{\max(1,\beta)}{\sum_{j=0}^{K}\tau_j}\sum_{k=0}^{K}\tau_k\E\left[\|z_k-x_k\|^2+\|d_k-\nabla\Phi(x_k)\|^2\right]= O_T\left(\frac{\sum_{k=0}^{K}\tau_k^2 + 1 + \delta_g^2\sum_{k=0}^{K}\tau_k}{\sum_{k=0}^{K}\tau_k}\right).
\end{align*}
The equality holds because of \eqref{ineq: newbound_d_phi} and \eqref{ineq: merit_sum}. 
\end{proof}
To ensure condition \eqref{ineq: quadratic_form}, we need the following technical result.
\begin{lemma}\label{lem: technical_lemma}
    There exist positive constants $\beta, \nu, \rho_1,...,\rho_T, \lambda, c$ such that for any positive constants $x, y, z_1,...,z_T$, we have
    \[
        (\beta - \lambda - 2\nu L_{y^*}^2-c)x^2 + (\nu - c)y^2 + \sum_{i=1}^{T}(\rho_i - c)z_i^2 - \tilde{C}xy - \tilde{C}\sum_{j=2}^{T}xz_i \geq 0
    \]
    implying that \eqref{ineq: quadratic_form} holds.
\end{lemma}
\begin{proof}
    Note that the above conclusion is equivalent to show that $A^{(T)}\succeq 0$, where $A^{(T)}=(a_{ij})\in \R^{(T+2)\times (T+2)}$ and its non-zero elements only include
    \[
        a_{11} = (\beta - \lambda - 2\nu L_{y^*}^2-c),\ a_{1i} = a_{i1} = -\frac{\tilde{C}}{2},\ a_{22} = \nu - c,\ a_{jj} = \rho_{j-2}-c,
    \]
    for all $2\leq i\leq T+2$ and $3\leq j\leq T+2$. In other words, we have
    \[
        A^{(T)}=
        \begin{pmatrix}
            &a_{11} &-\frac{\tilde{C}}{2}& -\frac{\tilde{C}}{2} &\cdots &-\frac{\tilde{C}}{2} \\
            &-\frac{\tilde{C}}{2} &\nu-c &0 &\cdots & 0 \\
            &-\frac{\tilde{C}}{2} &0 & \rho_1-c & \cdots & \vdots\\
            &\vdots & \vdots & \vdots & \ddots & 0\\ 
            &-\frac{\tilde{C}}{2} & 0 & \cdots & 0 &\rho_T-c
        \end{pmatrix}
    \]
    Defining $A_{k} = \det(A^{(k)})$, for any $1\leq k\leq T$, we know by induction that
    \begin{align*}
        &A_{k+1} = (\rho_{k+1} - c)A_k - \frac{\tilde{C}^2}{4}(\nu-c)(\rho_1-c)\cdots(\rho_k-c),\\
        &A_1 = (\rho_1-c)(a_{11}(\nu-c) - \frac{\tilde{C}^2}{4}) - \frac{\tilde{C}^2}{4}(\nu-c).
    \end{align*}
    Hence, for any $1\leq k\leq T$, we have
    \[
        A_k = \left[A_0 - \frac{\tilde{C}^2(\nu - c)}{4}\left(\sum_{i=1}^{k}\frac{1}{\rho_i-c}\right)\right]\cdot \prod_{i=1}^{k}\left(\rho_i-c\right),\ A_0 = a_{11}(\nu - c) - \frac{\tilde{C}^2}{4}.
    \]
    Hence, $A^{(T)}\succeq 0$ if and only if for any $0\leq k\leq T, A_k\geq 0$ and $a_{11}\geq 0$. One sufficient condition is to set
    \[
        \nu = \rho_1 = \cdots = \rho_T = 2c,\ \beta - \lambda = 4cL_{y^*}^2 + c + \frac{\tilde{C}^2(T+1)}{4c},\
    \]
    for any constant $c>0$. Then, it is trivial to verify that, for any $0\leq k\leq T$,
    \[
        a_{11} = \frac{\tilde{C}^2(T+1)}{4c} \geq 0,\ A_k = \frac{\tilde{C}^2(T-k)c^k}{4} \geq 0.
    \]
\end{proof}
 
\section{Simulation results}
\label{sec:experiments}
We now present our simulation results on the robust feature learning problem introduced in Section~\ref{sec:intro}. We consider solving the distributionally robust feature learning problem \eqref{eq:dro_minmax_form}, in which $Y$ is generated using multi-index model $Y = \sum_{i=1}^{50}\left(\omega_i\T X + c\cdot \sin(\omega_i\T X)\right) + \varepsilon$, a popular model in statistics and economics~\cite{hardle2004nonparametric,schmidt2020nonparametric}. Here, each $\omega_i$ is sampled from $\mathcal{N}(0, I_{100})$ and then normalized to a unit vector in $\R^{100}$, $\varepsilon$ is the noise sampled from normal distribution $\mathcal{N}(0, 0.01^2)$ for training data and $\mathcal{N}(0, 0.1^2)$ for testing. We set $c=1$ and $1.5$ in our experiments. Covariate $X\in \R^{100}$ and all the coordinates are independent and identically generated from the uniform distribution over $[0, 1]$. We consider a fully-connected neural network with two hidden layers, each of which has 50 nodes. The activation function is (smoothed) ReLU. The $\beta$ in \eqref{eq:dro_minmax_form} represents the weights in the last layer and the feature mapping $\Phi$ is the whole model without the last layer. 

\begin{figure*}[ht]
	\centering  
	\subfigure[]{\label{fig: 36_1sin}\includegraphics[width=0.48\textwidth]{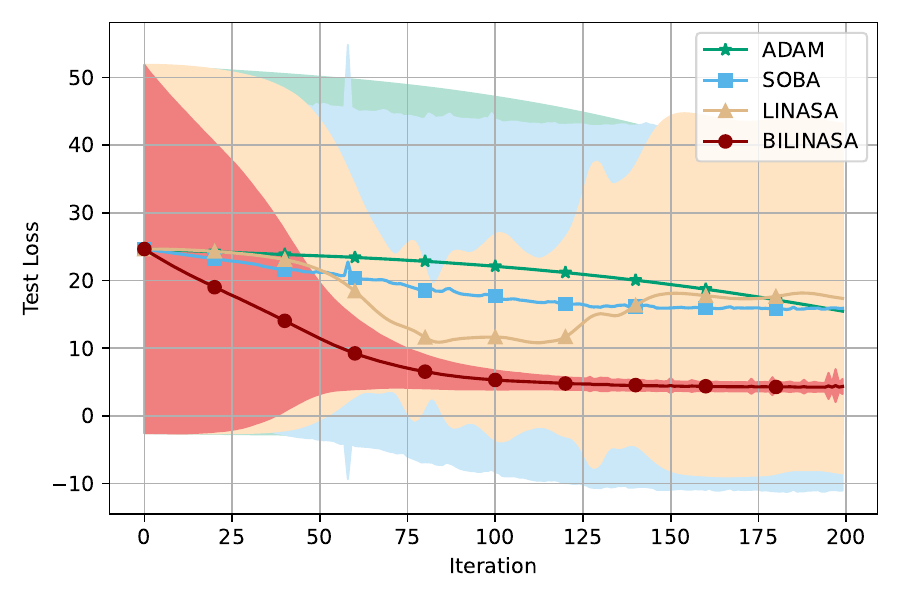}}
	\subfigure[]{\label{fig: 1545_1sin}\includegraphics[width=0.48\textwidth]{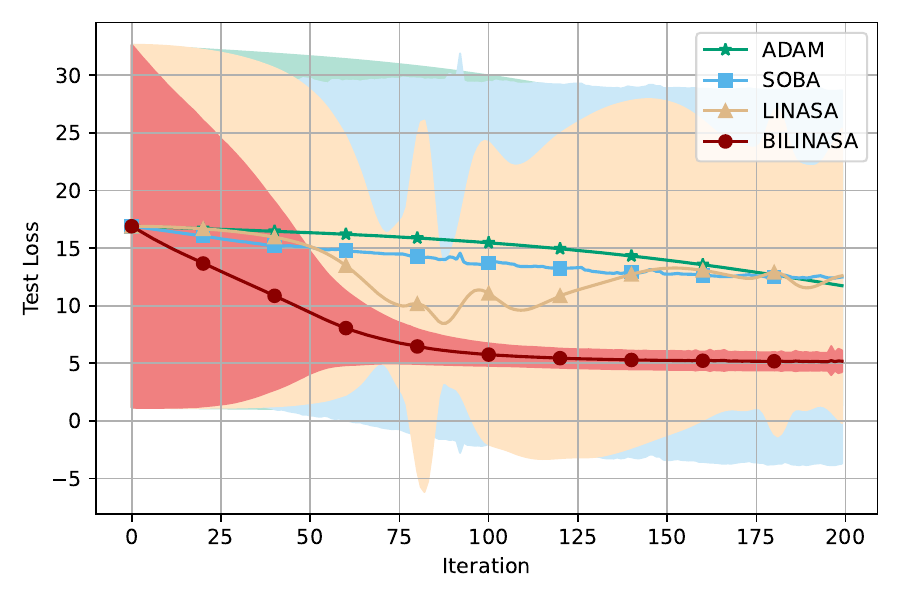}}
        \subfigure[]{\label{fig: 36_15sin}\includegraphics[width=0.48\textwidth]{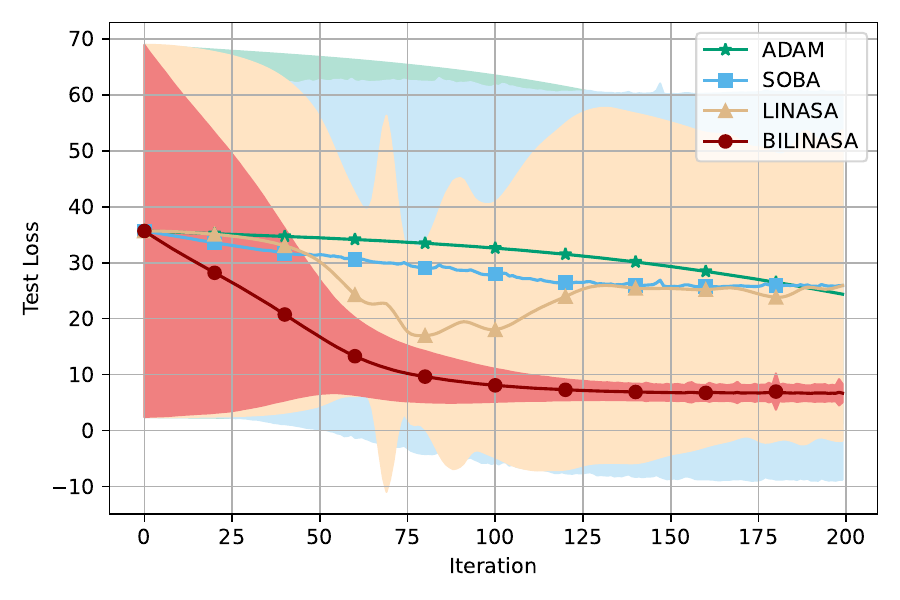}}
    \subfigure[]{\label{fig: 1545_15sin}\includegraphics[width=0.48\textwidth]{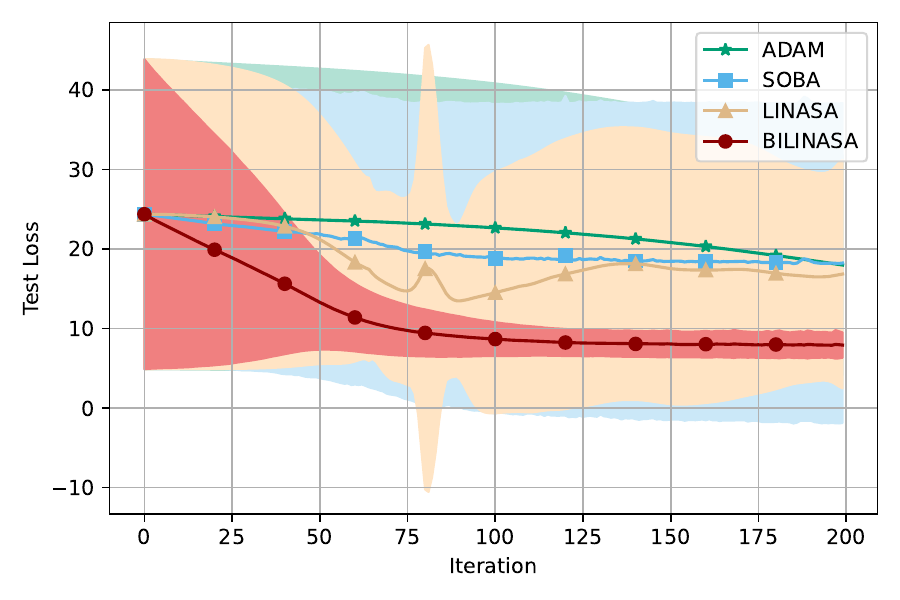}}
	\vspace{-0.2cm}
	\caption{Comparison between \texttt{Adam}, \texttt{SOBA}, \texttt{LINASA}, and \texttt{BILINASA}. Each curve represents the mean value of 100 trials and the shaded regions represent standard deviation. The first row corresponds to the results of the case $c=1$ and the other row corresponds to $c=1.5$. The first column has $(a, b) = (3, 6)$ and the other column has $(a, b) = (1.5, 4.5)$.}
        \label{Fig:main}
\end{figure*}

The reason behind the formulation \eqref{eq:dro_minmax_form} is to learn a robust feature mapping $\Phi$ that can generalize well on a test distribution $Q$. By reformulation \eqref{eq:droform} we know that this problem can be solved via \texttt{BiLiNASA}, i.e., Algorithm \ref{algo: bilinasa}. To compare the generalization performance of \texttt{BiLiNASA}, we also conduct experiments using \texttt{Adam} \cite{kingma2014adam}, \texttt{SOBA} \cite{dagreou2022framework}, and \texttt{LiNASA} \cite{balasubramanian2022stochastic}.
\begin{enumerate}
    \item For \texttt{Adam}, we choose the stepsize to be $10^{-4}$ and solve the following problem:
\begin{align*}
\min_{\Phi, \beta} ~~ \E_P [(Y  - \langle \beta, \Phi(X)\rangle)^2]
\end{align*}
directly.
\item  For \texttt{SOBA}, we set the stepsizes $\rho^t = \gamma^t = \frac{0.01}{\sqrt{K}}$ in Algorithm 1 of \cite{dagreou2022framework}, as suggested in Theorem 1 of \cite{dagreou2022framework}, to solve the following bi-level problem:
\begin{align*}
\min_{\Phi} ~~ \E_P \left[(Y  - \langle \beta, \Phi(X)\rangle)^2\right] \quad\text{s.t.}\quad \beta = \underset{\tilde \beta \in \mathbb{R}^p}{\argmin}~~\E_P\left[ (Y  - \langle \tilde \beta, \Phi(X)\rangle)^2 \right].
\end{align*}
\item For \texttt{LiNASA}, we set $\tau_k = \frac{0.1}{\sqrt{K}}, \beta = 3$ in Algorithm 2 of \cite{balasubramanian2022stochastic}, to solve the following problem:
\begin{align}\label{opt: linasa}
\resizebox{.85\hsize}{!}{$\underset{\Phi, \beta}{\min} \left\{ \E_P[(Y  - \langle \beta, \Phi(X)\rangle)^2] + \lambda (\E_P[\max(0,(Y  - \langle \beta, \Phi(X)\rangle)^2 - \E_P[(Y  - \langle \beta, \Phi(X)\rangle)^2])^2])^{1/2} \right\}$}.
\end{align}
\item For \texttt{BiLiNASA} we choose $\alpha = 0.01, M=\lfloor \log K\rfloor$ in Algorithm \ref{algo: nhe}, and $N=5, \beta = 30, \tau_k = \gamma_k = \frac{0.03}{\sqrt{K}}$ in Algorithm \ref{algo: bilinasa} to solve \eqref{eq:droform}:
\begin{align}\label{opt: bililinasa}
&\resizebox{.9\hsize}{!}{$\underset{\Phi}{\min} \left\{ \E_P[(Y  - \langle \beta, \Phi(X)\rangle)^2] + \lambda (\E_P[\max(0,(Y  - \langle \beta, \Phi(X)\rangle)^2 - \E_P[(Y  - \langle \beta, \Phi(X)\rangle)^2] )^2])^{1/2} \right\}$}\notag\\
\hspace{-0.3in}&\qquad\qquad\text{s.t.}~\beta = \underset{\tilde \beta \in \mathbb{R}^p}{\argmin}~~\E_P\left[ (Y  - \langle \tilde \beta, \Phi(X)\rangle)^2 \right].
\end{align}
\end{enumerate}
For~\eqref{opt: linasa} and~\eqref{opt: bililinasa}, we use a smoothed version of the $\max$ function for the experiments. In all algorithms the number of (outer-loop) iterations (i.e., $K$ in Algorithm \ref{algo: bilinasa}) is fixed to be 200. To evaluate the generalization performance of each algorithm, all entries of each test data $X$ are independent and identically sampled from the test distribution $Q$ which is set to be the Beta distribution over $[0, 1]$. Recall that the corresponding density function is given by
\begin{equation}\label{eq: beta_ab}
    p(x;a,b) = \frac{x^a(1-x)^b}{\int_0^1 u^a(1-u)^b du},
\end{equation}
where we choose $(a, b) = (3, 6)$ and $(1.5, 4.5)$ in our experiments. According to \cite{ruszczynski2006optimization, zhu2023distributionally}, $\lambda$ is chosen to be $1.65\times 10^{-3}$ when $(a, b) = (3, 6)$ and $1.89\times 10^{-3}$ when $(a, b) = (1.5, 4.5)$ in \eqref{eq: beta_ab}. We compare the algorithms via testing each of it on 100 trials, and the results are summarized in Figure \ref{Fig:main}. For all algorithms we evaluate the mean-squared losses on the same test dataset of size 1000. We plot the average of all the losses and the shaded regions represent standard deviation. From the result, we observe that the nested compositional bi-level  formulation solved using the proposed \texttt{BiLiNASA} method outperforms other formulations and algorithms in terms of test loss, which indicates that solving problem \ref{eq:droform} via Algorithm \ref{algo: bilinasa} indeed introduces robustness in the feature learning process. In particular, it is worth noting that the comparison between \texttt{BiLiNASA} and \texttt{LiNASA} indicates the superiority of reformulating the single level optimization problem in \eqref{opt: linasa}, which learns the features and regression coefficients jointly), as a bi-level one in \eqref{opt: bililinasa}, which learns the features robustly and the regression coefficients only using least-squares. 

\section{Conclusion}
In this paper, we study a class of problems at the intersection of bi-level and nested compositional optimization arising in several application domains including robust feature learning. This class of problems consist of bi--level problems in which the upper-level objective function has a nested composition structure
imposing additional challenge in controlling the error in estimating the hypergradient. We propose a novel stochastic approximation algorithm and establish its finite-time convergence analysis showing that it can achieve state-of-the-art complexity bounds (optimal up to an additional log factor) for bilevel nested compositional problems.

\bibliographystyle{abbrv}
\bibliography{references}

\begin{thebibliography}{10}

\bibitem{arbel2021amortized}
M.~Arbel and J.~Mairal.
\newblock Amortized implicit differentiation for stochastic bilevel
  optimization.
\newblock In {\em International Conference on Learning Representations}, 2021.

\bibitem{balasubramanian2022stochastic}
K.~Balasubramanian, S.~Ghadimi, and A.~Nguyen.
\newblock Stochastic multilevel composition optimization algorithms with
  level-independent convergence rates.
\newblock {\em SIAM Journal on Optimization}, 32(2):519--544, 2022.

\bibitem{BraMcG73}
J.~Bracken and J.~T. McGill.
\newblock Mathematical programs with optimization problems in the constraints.
\newblock {\em Operations Research}, 21(1):37--44, 1973.

\bibitem{chen2023bilevel}
L.~Chen, J.~Xu, and J.~Zhang.
\newblock On bilevel optimization without lower-level strong convexity.
\newblock {\em arXiv preprint arXiv:2301.00712}, 2023.

\bibitem{chen2022single}
T.~Chen, Y.~Sun, Q.~Xiao, and W.~Yin.
\newblock A single-timescale method for stochastic bilevel optimization.
\newblock In {\em International Conference on Artificial Intelligence and
  Statistics}, pages 2466--2488. PMLR, 2022.

\bibitem{chen2021closing}
T.~Chen, Y.~Sun, and W.~Yin.
\newblock Closing the gap: Tighter analysis of alternating stochastic gradient
  methods for bilevel problems.
\newblock {\em Advances in Neural Information Processing Systems},
  34:25294--25307, 2021.

\bibitem{chen2021solving}
T.~Chen, Y.~Sun, and W.~Yin.
\newblock Solving stochastic compositional optimization is nearly as easy as
  solving stochastic optimization.
\newblock {\em IEEE Transactions on Signal Processing}, 69:4937--4948, 2021.

\bibitem{chen2022decentralized}
X.~Chen, M.~Huang, and S.~Ma.
\newblock Decentralized bilevel optimization.
\newblock {\em arXiv preprint arXiv:2206.05670}, 2022.

\bibitem{chen2023decentralized}
X.~Chen, M.~Huang, S.~Ma, and K.~Balasubramanian.
\newblock Decentralized stochastic bilevel optimization with improved
  per-iteration complexity.
\newblock In {\em International Conference on Machine Learning}, 2023.

\bibitem{chen2023optimal}
X.~Chen, T.~Xiao, and K.~Balasubramanian.
\newblock Optimal algorithms for stochastic bilevel optimization under relaxed
  smoothness conditions.
\newblock {\em arXiv preprint arXiv:2306.12067}, 2023.

\bibitem{CouWan15}
N.~Couellan and W.~Wang.
\newblock Bi-level stochastic gradient for large scale support vector machine.
\newblock {\em Neurocomputing}, 153:300--308, 2015.

\bibitem{CouWan16}
N.~Couellan and W.~Wang.
\newblock On the convergence of stochastic bi-level gradient methods.
\newblock {\em preprint, optimization-online 13833}, 2016.

\bibitem{cui2020multicomposite}
Y.~Cui, Z.~He, and J.-S. Pang.
\newblock Multicomposite nonconvex optimization for training deep neural
  networks.
\newblock {\em SIAM Journal on Optimization}, 30(2):1693--1723, 2020.

\bibitem{dagreou2022framework}
M.~Dagr{\'e}ou, P.~Ablin, S.~Vaiter, and T.~Moreau.
\newblock A framework for bilevel optimization that enables stochastic and
  global variance reduction algorithms.
\newblock {\em Advances in Neural Information Processing Systems},
  35:26698--26710, 2022.

\bibitem{ermoliev1976methods}
Y.~M. Ermoliev.
\newblock {\em Methods of Stochastic Programming}.
\newblock Nauka, Moscow, 1976.

\bibitem{ghadimi2016mini}
S.~Ghadimi, G.~Lan, and H.~Zhang.
\newblock Mini-batch stochastic approximation methods for nonconvex stochastic
  composite optimization.
\newblock {\em Mathematical Programming}, 155(1-2):267--305, 2016.

\bibitem{ghadimi2020single}
S.~Ghadimi, A.~Ruszczynski, and M.~Wang.
\newblock A single timescale stochastic approximation method for nested
  stochastic optimization.
\newblock {\em SIAM Journal on Optimization}, 30(1):960--979, 2020.

\bibitem{ghadimi2018approximation}
S.~Ghadimi and M.~Wang.
\newblock Approximation methods for bilevel programming.
\newblock {\em arXiv preprint arXiv:1802.02246}, 2018.

\bibitem{HaBrGi92}
P.~Hansen, B.~Jaumard, and G.~Savard.
\newblock New branch-and-bound rules for linear bilevel programming.
\newblock {\em SIAM Journal on Scientific and Statistical Computing},
  13(5):1194--1217, 1992.

\bibitem{hardle2004nonparametric}
W.~H{\"a}rdle, M.~M{\"u}ller, S.~Sperlich, and A.~Werwatz.
\newblock {\em Nonparametric and semiparametric models}, volume~1.
\newblock Springer, 2004.

\bibitem{hong2023two}
M.~Hong, H.-T. Wai, Z.~Wang, and Z.~Yang.
\newblock A two-timescale stochastic algorithm framework for bilevel
  optimization: Complexity analysis and application to actor-critic.
\newblock {\em SIAM Journal on Optimization}, 33(1):147--180, 2023.

\bibitem{hu2021bias}
Y.~Hu, X.~Chen, and N.~He.
\newblock On the bias-variance-cost tradeoff of stochastic optimization.
\newblock {\em Advances in Neural Information Processing Systems},
  34:22119--22131, 2021.

\bibitem{hu2020biased}
Y.~Hu, S.~Zhang, X.~Chen, and N.~He.
\newblock Biased stochastic first-order methods for conditional stochastic
  optimization and applications in meta learning.
\newblock {\em Advances in Neural Information Processing Systems},
  33:2759--2770, 2020.

\bibitem{huang2022efficiently}
M.~Huang, X.~Chen, K.~Ji, S.~Ma, and L.~Lai.
\newblock Efficiently escaping saddle points in bilevel optimization.
\newblock {\em arXiv preprint arXiv:2202.03684v2}, 2023.

\bibitem{ji2021bilevel}
K.~Ji, J.~Yang, and Y.~Liang.
\newblock Bilevel optimization: Convergence analysis and enhanced design.
\newblock In {\em International Conference on Machine Learning}, pages
  4882--4892. PMLR, 2021.

\bibitem{kingma2014adam}
D.~P. Kingma and J.~Ba.
\newblock Adam: A method for stochastic optimization.
\newblock In {\em International Conference on Representational Learning}, 2015.

\bibitem{liu2022solving}
J.~Liu, Y.~Cui, and J.-S. Pang.
\newblock Solving nonsmooth and nonconvex compound stochastic programs with
  applications to risk measure minimization.
\newblock {\em Mathematics of Operations Research}, 47(4):3051--3083, 2022.

\bibitem{nesterov2013gradient}
Y.~Nesterov.
\newblock Gradient methods for minimizing composite functions.
\newblock {\em Mathematical programming}, 140(1):125--161, 2013.

\bibitem{nesterov2018lectures}
Y.~Nesterov.
\newblock {\em Lectures on convex optimization}, volume 137.
\newblock Springer, 2018.

\bibitem{pan2010survey}
S.~J. Pan and Q.~Yang.
\newblock A survey on transfer learning.
\newblock {\em IEEE Transactions on knowledge and data engineering},
  22(10):1345--1359, 2010.

\bibitem{qu2017harnessing}
G.~Qu and N.~Li.
\newblock Harnessing smoothness to accelerate distributed optimization.
\newblock {\em IEEE Transactions on Control of Network Systems},
  5(3):1245--1260, 2017.

\bibitem{ruszczynski2020convergence}
A.~Ruszczy{\'n}ski.
\newblock Convergence of a stochastic subgradient method with averaging for
  nonsmooth nonconvex constrained optimization.
\newblock {\em Optimization Letters}, 14(7):1615--1625, 2020.

\bibitem{rusz20}
A.~Ruszczynski.
\newblock A stochastic subgradient method for nonsmooth nonconvex multilevel
  composition optimization.
\newblock {\em SIAM Journal on Control and Optimization}, 59(3):2301--2320,
  2021.

\bibitem{ruszczynski2006optimization}
A.~Ruszczy{\'n}ski and A.~Shapiro.
\newblock Optimization of convex risk functions.
\newblock {\em Mathematics of operations research}, 31(3):433--452, 2006.

\bibitem{schmidt2020nonparametric}
J.~Schmidt-Hieber.
\newblock Nonparametric regression using deep neural networks with {ReLU}
  activation function.
\newblock {\em Annals of statistics}, 48(4):1875--1897, 2020.

\bibitem{shapiro2021lectures}
A.~Shapiro, D.~Dentcheva, and A.~Ruszczynski.
\newblock {\em Lectures on stochastic programming: modeling and theory}.
\newblock SIAM, 2021.

\bibitem{ShiLuZha05}
C.~Shi, J.~Lu, and G.~Zhang.
\newblock An extended {Kuhn-Tucker} approach for linear bilevel programming.
\newblock {\em Applied Mathematics and Computation}, 162(1):51--63, 2005.

\bibitem{sugiyama2012machine}
M.~Sugiyama and M.~Kawanabe.
\newblock {\em Machine learning in non-stationary environments: Introduction to
  covariate shift adaptation}.
\newblock MIT press, 2012.

\bibitem{von1952theory}
H.~von Stackelberg.
\newblock {\em The Theory of the Market Economy}.
\newblock William Hodge, 1952.

\bibitem{wang2017stochastic}
M.~Wang, E.~X. Fang, and B.~Liu.
\newblock Stochastic compositional gradient descent: Algorithms for minimizing
  compositions of expected-value functions.
\newblock {\em Mathematical Programming}, 161(1-2):419--449, 2017.

\bibitem{WaLiFa17}
M.~Wang, J.~Liu, and E.~X. Fang.
\newblock Accelerating stochastic composition optimization.
\newblock {\em Journal of Machine Learning Research}, 18:1--23, 2017.

\bibitem{yang2021provably}
J.~Yang, K.~Ji, and Y.~Liang.
\newblock Provably faster algorithms for bilevel optimization.
\newblock {\em Advances in Neural Information Processing Systems},
  34:13670--13682, 2021.

\bibitem{yang2019multi}
S.~Yang, M.~Wang, and E.~X. Fang.
\newblock Multilevel stochastic gradient methods for nested composition
  optimization.
\newblock {\em SIAM Journal on Optimization}, 29(1):616--659, 2019.

\bibitem{yang2022decentralized}
S.~Yang, X.~Zhang, and M.~Wang.
\newblock Decentralized gossip-based stochastic bilevel optimization over
  communication networks.
\newblock {\em Advances in Neural Information Processing Systems}, 35:238--252,
  2022.

\bibitem{zhang2021multilevel}
J.~Zhang and L.~Xiao.
\newblock Multilevel composite stochastic optimization via nested variance
  reduction.
\newblock {\em SIAM Journal on Optimization}, 31(2):1131--1157, 2021.

\bibitem{zhu2023distributionally}
L.~Zhu, M.~G{\"u}rb{\"u}zbalaban, and A.~Ruszczy{\'n}ski.
\newblock Distributionally robust learning with weakly convex losses:
  Convergence rates and finite-sample guarantees.
\newblock {\em arXiv preprint arXiv:2301.06619}, 2023.

\end{thebibliography}
\appendix
\section{An issue in Neumann series analysis in prior works}\label{sec: issue}
In this section, we point out the issue in  \cite{ghadimi2018approximation,hong2023two,chen2021closing} on estimating the Hessian inverse using Neumann series. Using our notation, the following inequality is used in the above works; see (3.76) in \cite{ghadimi2018approximation}, Lemma 11 in arXiv version 2 of \cite{hong2023two}, and Lemma 5 in \cite{chen2021closing}:

\begin{equation}\label{ineq: hyper_var}
    \E\left[\|w_k - \E\left[w_k|\F_k\right]\|^2\right]\leq \sigma_w^2,\quad \E\left[\|w_k\|^2\right]\leq \tilde{\sigma}_w^2,
\end{equation}
where $w_k$ is the hypergradient estimate, which is constructed by Neumann series based approach, and $\sigma_w$ and $\tilde{\sigma}_w^2$ are constants that are independent of $M$. 
It essentially means that the variance of the hypergradient estimate can be bounded by some given constants that are independent of $M$. Below, we provide a counter example refuting the claim that  the constant~$\sigma_w^2$ and $\tilde{\sigma}_w^2$ are independent of $M$. 

Suppose we are going to use Neumann series based approach to estimate $A^{-1}$, where $\mu I\preceq A \preceq L I$ for some constants $0<\mu<L$. The process is:
\begin{itemize}
    \item Fix an integer $M>0$. Sample $p$ from $\{0,1,...,M-1\}$ uniformly at random.
    
    \item Set $X = \frac{M}{L}\prod_{i=1}^{p}(I - \frac{1}{L}A_i)$ as the estimate for $A$, where each $A_i$ satisfies $\E\left[A_i\right] = A, \E\left[\|A_i -A\|^2\right]\leq \sigma^2$ and $A_i$ is independent of $A_j$ for $j\neq i$. $X = \frac{M}{L}I$ if $p=0$. 
\end{itemize}
Note that if we set $A_i = A = 1$ (noiseless scalar), then we have:
\begin{align*}
    X &= \frac{M}{L}\bigg(1-\frac{1}{L}\bigg)^p \text{ with probability } \frac{1}{M}, \\
       \E\left[X\right] &= \frac{1}{M}\sum_{p=0}^{M-1}\frac{M}{L}\bigg(1-\frac{1}{L}\bigg)^p = 1 - \bigg(1-\frac{1}{L}\bigg)^M.
\end{align*}
Thus, for $M>L$ we have
\begin{align*}
    &\E\left[\|X- \E\left[X\right]\|^2\right] > \frac{1}{M}\bigg\|\frac{M}{L} - \bigg(1 - \bigg(1-\frac{1}{L}\bigg)^M\bigg)\bigg\|^2 \\
    =\, &\frac{1}{M}\bigg(\frac{M}{L} - 1 + \bigg(1-\frac{1}{L}\bigg)^M\bigg)^2 > \frac{1}{M}\bigg(\frac{M}{L} - 1\bigg)^2,
\end{align*}
where in the first inequality we just consider the case when $p=0$, and the last inequality uses $M>L$. Note that this means the upper bound of the variance of $X$ must depend on $M$ if we use this process.

As a consequence, for the algorithms in the above works, an additional $\log$ factor is introduced in the overall sample complexity. Indeed, as we show in our Theorem~\ref{thm:mainresult}, we need to pick $M=\Theta(\log K)$, with $K$ denoting the number of outer-loops.


\end{document}